\documentclass[10pt]{amsart}
\usepackage{latexsym}
\usepackage{indentfirst}
\usepackage{amssymb,amsfonts,amsmath,amsthm}
\usepackage{graphicx}
\usepackage{mathrsfs}
\usepackage{array}
\usepackage{color}
\usepackage{epstopdf}
\usepackage[all]{xy}
\usepackage{hyperref}
\usepackage{url}

\newtheorem{thm}{Theorem}[section]
\newtheorem{prop}[thm]{Proposition}
\newtheorem{lem}[thm]{Lemma}

\newtheorem{lem-def}[thm]{Lemma-Definition}
\newtheorem{cor}[thm]{Corollary}
\theoremstyle{remark}
\newtheorem{exam}[thm]{Example}

\newtheorem{rmk}{Remark}[section]
\theoremstyle{definition}
\newtheorem{dfn}[thm]{Definition}

\numberwithin{equation}{section}

\newcommand{\frakm}{{\mathfrak m}}

\newcommand{\frakB}{{\mathfrak B}}

\newcommand{\frakS}{{\mathfrak S}}

\newcommand{\frakX}{{\mathfrak X}}

\newcommand{\bA}{{\mathbb A}}

\newcommand{\bC}{{\mathbb C}}

\newcommand{\bL}{{\mathbb L}}
\newcommand{\bM}{{\mathbb M}}

\newcommand{\bQ}{{\mathbb Q}}

\newcommand{\bZ}{{\mathbb Z}}

\newcommand{\calE}{{\mathcal E}}

\newcommand{\calL}{{\mathcal L}}
\newcommand{\calM}{{\mathcal M}}

\newcommand{\calO}{{\mathcal O}}

\newcommand{\rA}{{\mathrm A}}

\newcommand{\rC}{{\mathrm C}}

\newcommand{\rH}{{\mathrm H}}

\newcommand{\rM}{{\mathrm M}}

\newcommand{\rW}{{\mathrm W}}

\newcommand{\Zp}{{\bZ_p}}

\newcommand{\Cp}{{\bC_p}}

\newcommand{\Ainf}{{\mathrm{A_{inf}}}}

\newcommand{\AAinf}{{\bA_{\mathrm{inf}}}}            

\newcommand{\Spa}{{\mathrm{Spa}}}           
\newcommand{\Spf}{{\mathrm{Spf}}}           
\newcommand{\Spec}{{\mathrm{Spec}}}         

\newcommand{\colim}{{\mathrm{colim}}}

\newcommand{\Ker}{{\mathrm{Ker}}}           
\newcommand{\Ima}{{\mathrm{Im}}}            

\newcommand{\End}{{\mathrm{End}}}           
\newcommand{\ev}{{\mathrm{ev}}}             

\newcommand{\Crys}{{\mathrm{Crys}}}

\newcommand{\et}{{\mathrm{\text{\'e}t}}}    
\newcommand{\Et}{\mathrm{\text{\'E}t}}
\newcommand{\proet}{{\mathrm{pro\text{\'e}t}}}

\newcommand{\perf}{\mathrm{perf}}           

\usepackage{relsize}
\usepackage[bbgreekl]{mathbbol}
\DeclareSymbolFontAlphabet{\mathbb}{AMSb} 
\DeclareSymbolFontAlphabet{\mathbbl}{bbold}
\newcommand{\Prism}{{\mathlarger{\mathbbl{\Delta}}}}

\newcommand{\id}{{\mathrm{id}}}             

\newcommand{\IF}{{\mathrm{if}}}

\begin{document}
	\title{Relative $(\varphi,\Gamma)$-modules and prismatic $F$-crystals}
	
	\author{Yu Min}
	\thanks{Y.M. $\bf Email$: yu.min@amss.ac.cn. $\bf Address$: Morningside Center of Mathematics No.505, Chinese Academy of Sciences, ZhongGuancun East Road 55, Beijing, 100190, China.}

	\author{Yupeng Wang}
	\thanks{Y.W. $\bf Email$: wangyupeng@amss.ac.cn. $\bf Address$: Morningside Center of Mathematics No.505, Chinese Academy of Sciences, ZhongGuancun East Road 55, Beijing, 100190, China.}

	\date{\today}
	\maketitle

	\begin{abstract}
		In this paper, we prove that for any $p$-adic smooth separated formal scheme $\frakX$, the category of prismatic $F$-crystals over $\calO_{\Prism}[\frac{1}{I}]^{\wedge}_p$ is equivalent to the category of \'etale $\bZ_p$-local systems on the generic fiber of $\frakX$. We then compare the cohomology of the corresponding coefficients.
	\end{abstract}
	
	\tableofcontents
	
	\section{Introduction}\label{Sec-Introduction}
	In rational $p$-adic Hodge theory, there is the classical theorem due to Colmez--Fontaine that the category ${\rm Rep}_{\bQ_p}^{\rm cris}(G_K)$ of crystalline representations of the absolute Galois group $G_K$ of a $p$-adic field $K$, is equivalent to the category ${\rm MF}^{\varphi}_K$ of weakly admissible filtered $\varphi$-modules. It is tempting to ask for an integral version of this result. In fact, one also hopes to use some semilinear algebras to describe crystalline $\bZ_p$-representations. There are two ways to deal with this question. One is the cyclotomic case: there is the category of $(\varphi, \Gamma)$-modules which is equivalent to the category ${\rm Rep}_{\bZ_p}^{\rm cris}(G_K)$. The other is the Kummer case: there are the theories of Breuil modules and Breuil--Kisin modules which however can not describe the whole category ${\rm Rep}_{\bZ_p}^{\rm cris}(G_K)$ in general.  Combining these two points of view, there is the theory of $({\varphi, \hat G})$-modules which also describes the whole category ${\rm Rep}_{\bZ_p}^{\rm cris}(G_K)$.

	On the other hand, integral $p$-adic Hodge theory has seen great developements recently. In \cite{BS}, Bhatt and Scholze introduced the theory of prismatic cohomology, which unifies its precedents: the $\AAinf$-cohomology theory in \cite{BMS-a} and the Breuil--Kisin cohomology in \cite{BMS-b}. The latter two originate from the seeking of a cohomological construction of the Breuil--Kisin modules. It turns out that the prismatic formalism contains finer integral information than the Breuil--Kisin theory.  In \cite{BS1}, Bhatt and Scholze proved that the category of prismatic $F$-crystals over $\calO_{\Prism}$ (resp. ${\calO_{\Prism}[\frac{1}{I}]^{\wedge}_p}$) on the absolute prismatic site $(\calO_K)_{\Prism}$, where $\calO_K$ is the ring of integers in a $p$-adic field $K$, is equivalent to the category ${\rm Rep}_{\bZ_p}^{\rm cris}(G_K)$ (resp. ${\rm Rep}_{\bZ_p}(G_K)$). The equivalence between the category of prismatic $F$-crystals and the category ${\rm Rep}_{\bZ_p}(G_K)$ has also been proved in \cite{Wu}.
	
	In this paper, we want to generalize the work of Wu to the relative case, i.e. compare prismatic $F$-crystals and relative $(\varphi,\Gamma)$-modules. More presicely, we prove the following theorem.
	
	\begin{thm}[Theorem \ref{local system}]
		Let $\frakX$ be a separated $p$-adic smooth formal scheme over $\calO_K$ and $X$ be its adic generic fiber. Then there is an equivalence of categories
		\[
		{\rm Crys}(\calO_{\Prism}[\frac{1}{I}]^{\wedge}_p, \varphi)\xrightarrow{\simeq} {\rm LS}(X_{\rm \acute et}, \bZ_p),
		\]
		where the right one is the category of \'etale $\bZ_p$-local systems on $X$.
	\end{thm}
	The proof of this theorem uses relative $(\varphi,\Gamma)$-modules as intermediate objects, which are certain pro-\'etale sheaves on the generic fiber by the work of Kedlaya--Liu \cite{KL}. We then give an equivalence between the category of prismatic $F$-crystals and the category of relative $(\varphi,\Gamma)$-modules following the idea of Wu.
	\begin{rmk}
		In fact, Bhatt and Scholze proved a more general result than the above theorem in \cite{BS1}. Their results apply to crystals in perfect complex and $p$-adic formal schemes, which are not necessarily smooth. Compared to their proof,  the proof of our theorem involves a concrete construction of the equivalence. We hope the readers can still find this concrete description useful.
	\end{rmk}
	
	\begin{rmk}
		It seems natural to hope for a relative version of the result of Bhatt and Scholze concerning the crsytalline representations. Namely, we hope there is an equivalence between the category of prismatic $F$-crystals over $\calO_{\Prism}$ on the absolute prismatic site $\frakX_{\Prism}$ and the category of ``crystalline $\bZ_p$-local systems" on the rigid generic fiber of $\frakX$. Note that the category of ``crystalline $\bZ_p$-local systems" are only well-defined in the unramified case at the moment. Based on the idea of Bhatt--Scholze, one might need a description of crystalline $\bQ_p$-local systems via certain relative weakly admissible filtered $\varphi$-modules. But there seems still no satisfying relative theory of ``weakly admissible implies admissible", cf.\cite{Moon}.
		
	\end{rmk}
	
	Our second main theorem is an  \'etale comparison with coefficients.
	\begin{thm}[Theorem \ref{global etale}]
		Let $(A,I)$ be a perfect prism such that $A/I$ contains all the $p$-power roots of unity and $\frakX$ be a separated $p$-adic smooth formal scheme over $A/I$. For any $\calM$ which is  a prismatic $F$-crystal over $\calO_{\Prism}[\frac{1}{I}]^{\wedge}_p$ with corresponding $\Zp$-local system $\calL$ on $X_{\rm \acute et}$, there is a quasi-isomorphism
		\[
		R\Gamma((\frakX/A)_{\Prism}, \calM)^{\varphi=1}\simeq R\Gamma(X_{\rm \acute et}, \calL).
		\]
	\end{thm}
	
	The proof of this theorem also depends on an intermediate object, the perfect prismatic site, which appears already in the \'etale comparison with constant coefficients in \cite{BS}. By the same argument, one can also get the following theorem.
	
	\begin{thm}[Theorem \ref{O_K}]
	    	Let $\frakX$ be a separated $p$-adic smooth formal scheme over $\calO_K$ and $\calM$ be a prismatic $F$-crystal over $\calO_{\Prism}[\frac{1}{I}]^{\wedge}_p$ on $\frakX_{\Prism}$ with corresponding $\Zp$-local system $\calL$ on $X_{\rm \acute et}$. Then there is a quasi-isomorphism
		\[
		R\Gamma((\frakX)_{\Prism}, \calM)^{\varphi=1}\simeq R\Gamma(X_{\rm \acute et}, \calL).
		\]
	\end{thm}
Note that this theorem is a generalization of the \'etale comparison in \cite{BS} even in the case of constant coefficients. In particular, it shows that we need to consider cohomology of absolute prismatic site instead of relative one when we study the case of imperfect prisms.
	
	\begin{rmk}
		The above theorems should be regarded as a part of the whole comparison theory with coefficients as in \cite{BS}.
		In \cite{MT}, Morrow and Tsuji have proved an equivalence between the category of relative Breuil--Kisin--Fargues modules and the category of prismatic $F$-crystals with the base prism $(\Ainf, (\Ker(\theta)))$. In particular, they have constructed explicitly the $\bZ_p$-local systems, the vector bundles with integrable connections and the crystalline $F$-crystals out of relative Breuil--Kisin--Fargues modules (see also \cite{BS1} for some similar statements). They have also got some comparison results with these coefficients. Note that these coefficient objects are not constructed directly out of prismatic $F$-crystals and it seems that Morrow--Tsuji's arguments could not deal with imperfect prisms like the Breuil--Kisin prism $(\frakS, (E))$. We plan to study the case of imperfect prisms in a future work.
	\end{rmk}

	\section*{Acknowledgments}
	We would like to thank Zhiyou Wu for answering our questions. We also want to thank Heng Du, Shizhang Li, Ruochuan Liu, Matthew Morrow and Takeshi Tsuji for their valuable comments on the earlier drafts of this work.
	
	\section{Local construction}\label{local}
	Let $\Ainf$ be the period ring of infinitesimal deformation due to Fontaine and $\theta:\Ainf\rightarrow\calO:=\calO_{\Cp}$ be the canonical surjection. Fix a compatible system of $p$-power roots $\{\zeta_{p^n}\}_n$. Then $\Ker(\theta)$ is a principal ideal generated by $\xi = \frac{\mu}{\varphi^{-1}(\mu)}$ where $\mu = [\epsilon]-1$ and $\epsilon=(1, \zeta_p,\zeta_{p^2}\cdots)$.
	
	Let $\Spf(R)$ be a $p$-adic smooth formal scheme over $\calO$ with a framing $\Box: \calO\langle\underline T^{\pm 1}\rangle\rightarrow R$, i.e. an \'etale morphism over $\calO$. Let $(R/\Ainf)_{\Prism}$ be the prismatic site of $\Spf(R)$ over $(\Ainf,(\xi))$ of bounded prisms with the structure sheaf $\calO_{\Prism}$ and $(R/\Ainf)^{\perf}_{\Prism}$ be the site of perfect prisms over $(\Ainf,(\xi))$. Denote $\calO_{\Prism}[\frac{1}{\xi}]^{\wedge}_p$ the $p$-adic completion of $\calO_{\Prism}[\frac{1}{\xi}]$ and then for any $\frakB = (\Spf(R)\leftarrow\Spf(B/\xi B)\rightarrow \Spf(B))$ in $(R/\Ainf)_{\Prism}$, we have ${\calO_{\Prism}[\frac{1}{\xi}]^{\wedge}_p}(\frakB)=\widehat{B[\frac{1}{\xi}]}$ (cf.\cite{Wu})).
	
	\begin{rmk}\label{Rmk-Identify prismatic sites}
		The site $(R/\Ainf)_{\Prism}$ can be viewed as the absolute prismatic site $(R)_{\Prism}$ over $R$ of bounded prisms (\cite[Remark 4.7]{BS}). In fact, for any prism $(A,I)$ endowed with a morphism $R\rightarrow A/I$, the composition $\calO\rightarrow R\rightarrow A/I$ lifts uniquely to a morphism of prisms $(\Ainf,(\xi))\rightarrow (A,I)$ by the deformation theory as $\calO$ is perfectoid.
	\end{rmk}
	
	If we fix an $\Ainf$-lifting $\Ainf\langle\underline T^{\pm 1}\rangle$ (with the $(p,\xi)$-adic topology) of $\calO\langle\underline T^{\pm 1}\rangle$, then $R$ also admits an $\Ainf$-lifting $A^{\Box}_{\inf}(R)$ together with a formally \'etale morphism $\Ainf\langle\underline T^{\pm 1}\rangle\rightarrow A^{\Box}_{\inf}(R)$ lifting $\Box$, which is uniquely determined by the given framing $\Box$. Now, we equip $\Ainf\langle\underline T^{\pm 1}\rangle$ with the $\varphi$-action given by $\varphi(T_i)=T_i^p$ for $1\leq i\leq d$ compatible with the usual $\varphi$-action on $\Ainf$, and then such a $\varphi$-action extends uniquely to $A^{\Box}_{\inf}(R)$.
	
	\begin{exam}\label{Exam-some prisms}
		$(1)$~ The prism$(A^{\Box}_{\inf}(R), (\xi))$ is an object in $(R/\Ainf)_{\Prism}$, since $A^{\Box}_{\inf}(R)$ is formally smooth over $\Ainf$.
		
		$(2)$~ Let $(A^{\Box}_{\inf}(R),(\xi))^{\perf}$ be the perfection of $(A^{\Box}_{\inf}(R),(\xi))$ (\cite[Lemma 3.9]{BS}). Then $(A^{\Box}_{\inf}(R),(\xi))^{\perf}$ is a cover of the final object of the topos ${\rm Shv}((R/\Ainf)_{\Prism}^{\rm perf})$ (see Lemma \ref{Lem-Special cover}); the same holds for $(A^{\Box}_{\inf}(R),(\xi))$.
	\end{exam}
	
	\begin{dfn}\label{Dfn-Prismatic F-Crystal}
		$(1)$~ By a \emph{prismatic $F$-crystal} over $\calO_{\Prism}$ on $(R/\Ainf)_{\Prism}$, we mean a sheaf of $\calO_{\Prism}$-module $\calM$ endowed with an isomorphism $(\varphi^{*}\calM)[\frac{1}{\xi}]\xrightarrow{\simeq}\calM[\frac{1}{\xi}]$, such that for any object $\frakB = (\Spf(R)\leftarrow\Spf(B/\xi B)\rightarrow \Spf(B))$, $\calM(\frakB)$ is a finite projective $B=\calO_{\Prism}(\frakB)$-module and such that for any morphism \[\frakB'=(\Spf(R)\leftarrow\Spf(B'/\xi B')\rightarrow\Spf(B'))\rightarrow \frakB\]
		in $(R/\Ainf)_{\Prism}$, the induced morphism $\calM(\frakB)\otimes_BB'\rightarrow\calM(\frakB')$ is an isomorphism of $\varphi$-modules.
		
		$(2)$ By a \emph{prismatic $F$-crystal} over $\calO_{\Prism}[\frac{1}{\xi}]^{\wedge}_p$ on $(R/\Ainf)_{\Prism}$, we mean a sheaf of ${\calO_{\Prism}[\frac{1}{\xi}]^{\wedge}_p}$-module $\calM$ endowed with an isomorphism $\varphi_{\calM}:(\varphi^*\calM)\xrightarrow{\simeq} \calM$, such that for any object $\frakB$, $\calM(\frakB)$ is a finite projective $\widehat{B[\frac{1}{\xi}]}=\calO_{\Prism}[\frac{1}{\xi}]^{\wedge}_p(\frakB)$-module such that for any morphism $\frakB'\rightarrow \frakB$ in $(R/\Ainf)_{\Prism}$, the induced morphism $\calM(\frakB)\otimes_{\widehat{B[\frac{1}{\xi}]}}\widehat{B'[\frac{1}{\xi}]}\rightarrow\calM(\frakB')$ is an isomorphism of $\varphi$-modules.
		
	\end{dfn}
	
	\begin{dfn}\label{Dfn-Etale phi-Gamma modules}
		Let $R$ be a ring equipped with a continuous endomorphism $\varphi:R\rightarrow R$ and $\Gamma$ be a topological group.
		
		$(1)$~ By a \emph{$\varphi$-module} over $R$, we mean an $R$-module $M$ together with a $\varphi$-semi-linear continuous morphism $\varphi_M:M\rightarrow M$. A $\varphi$-module $M$ is called \emph{\'etale}, if it is finite projective and the linearization $\varphi^{\ast}M\rightarrow M$ of $\varphi_M$ is an isomorphism of $R$-modules. Let $\Phi\rM(R)$ and $\Et\Phi\rM(R)$ be the categories of $\varphi$-modules and \'etale $\varphi$-modules over $R$, respectively.
		
		$(2)$~ Assume, moreover, $R$ is endowed with a continuous $\Gamma$-action commuting with $\varphi$. By a \emph{$(\varphi,\Gamma)$-module} over $R$, we mean a $\varphi$-module $M$ together with a semi-linear continuous $\Gamma$-action commuting with $\varphi$. A $(\varphi,\Gamma)$-module $M$ is call \emph{\'etale}, if it is an \'etale $\varphi$-module. Let $\Phi\Gamma\rM(R)$ and $\Et\Phi\Gamma\rM(R)$ be the categories of $(\varphi,\Gamma)$-modules and \'etale $(\varphi,\Gamma)$-modules over $R$, respectively.
	\end{dfn}
	
	Now, we consider the category of prismatic $F$-crystals over $\calO_{\Prism}[\frac{1}{\xi}]^{\wedge}_p$ on $(R/\Ainf)_{\Prism}$.
	
	We first give another description of $(A^{\Box}_{\inf}(R),\xi)^{\perf}$. Let $R_{\infty}= R\widehat \otimes_{\calO\langle\underline T^{\pm 1}\rangle}\calO\langle\underline T^{\pm \frac{1}{p^{\infty}}}\rangle$, which is a perfectoid $\calO$-algebra. Thus, there is a unique $\Ainf$-lifting $\Ainf(R_{\infty})$ of $R_{\infty}$ together with a unique morphism $(\Ainf,(\xi))\rightarrow(\Ainf(R_{\infty}),(\xi))$ of prisms lifting the composition $\calO\rightarrow R\rightarrow R_{\infty}$. One can easily check that $\Ainf(R_{\infty})$ is the $(p,\xi)$-adic completion of $\Ainf\langle\underline T^{\pm \frac{1}{p^{\infty}}}\rangle\otimes_{\Ainf\langle\underline T^{\pm 1}\rangle}A^{\Box}_{\inf}(R)$, where $T_i$ is identified with $[T_i^{\flat}]$ for $T_i^{\flat} = (T_i, T_i^{\frac{1}{p}}, \cdots)\in R^{\flat}_{\infty}$ and that the morphism $A^{\Box}_{\inf}(R)\rightarrow \Ainf(R_{\infty})$ carrying $T_i$ to $T_i$ of $\Ainf$-algebras preserves $\delta$-structures and induces a morphism of prisms
	\[(A^{\Box}_{\inf}(R),(\xi))\rightarrow (\Ainf(R_{\infty}),(\xi)).\]
	
	\begin{lem}\label{Lem-Perfection}
		The above morphism induces an isomorphism
		\[(A^{\Box}_{\inf}(R),(\xi))^{\perf}\rightarrow (\Ainf(R_{\infty}),(\xi)).\]
	\end{lem}
	\begin{proof}
		Denote $(A_{\infty},(\xi)):=(A^{\Box}_{\inf}(R),(\xi))^{\perf}$. By \cite[Lemma 3.9]{BS}, $A_{\infty}$ is the $(p,\xi)$-adic completion of $\colim_{\varphi}A^{\Box}_{\inf}(R)$ and is the initial object among all perfect prisms over $(A^{\Box}_{\inf}(R),(\xi))$. In particular, there is a unique morphism $i:(A_{\infty},(\xi))\rightarrow (\Ainf(R_{\infty}),(\xi))$. We shall construct the inverse of $i$ to complete the proof.
		
		When $R = \calO\langle\underline T^{\pm 1}\rangle$, it is straightforward to check $A_{\infty} = \Ainf\langle\underline T^{\frac{1}{p^{\infty}}}\rangle$ and in particular, the lemma holds in this case. For the general case, as $R$ is \'etale over $\calO\langle\underline T^{\pm 1}\rangle$, there is a unique morphism
		\[j: A^{\Box}_{\inf}(R)\widehat \otimes_{\Ainf\langle\underline T^{\pm 1}\rangle}\Ainf\langle\underline T^{\frac{1}{p^{\infty}}}\rangle=:A_{\infty}'\rightarrow A_{\infty}\]
		once one regards $(A_{\infty},(\xi))$ as a perfect prism over $(\Ainf\langle\underline T^{\pm 1}\rangle,(\xi))$. According to the $(p,\xi)$-\'etaleness of $A^{\Box}_{\inf}(R)$ over $\Ainf\langle\underline T^{\pm 1}\rangle$, $(A_{\infty}',(\xi))$ is a prism. By construction of $A_{\infty}'$, its reduction modulo $\xi$ is the $p$-adical completion of $R\otimes_{\calO\langle\underline T^{\pm 1}\rangle}\calO\langle\underline T^{\pm \frac{1}{p^{\infty}}}\rangle = R_{\infty}$. By the equivalence in \cite[Theorem 3.10]{BS}, we deduce $(A_{\infty}',\xi)$ is a perfect prism and is isomorphic to $(\Ainf(R_{\infty}),(\xi))$.
		
		Thanks to the universal property of $(A_{\infty},(\xi))$, the composite $j\circ i$ is the identity morphism. It remains to check $i\circ j$ is also the identity morphism. By \cite[Theorem 3.10]{BS}, it is enough to verify this modulo $\xi$. However, by construction, we have the following commutative diagram of $\Ainf$-algebras
		\begin{equation*}
			\xymatrix@C=0.45cm{
				\Ainf\langle\underline T^{\pm \frac{1}{p^{\infty}}}\rangle\ar[d]\ar@{=}[rr]&&\Ainf\langle\underline T^{\pm \frac{1}{p^{\infty}}}\rangle\ar[d]\\
				A_{\infty}'\ar[r]^{j}&A_{\infty}\ar[r]^{i}& A_{\infty}'
			}.
		\end{equation*}
		Modulo $\xi$, we see that the morphism $i\circ j: R_{\infty}\rightarrow R_{\infty}$ of $R$-algebras preserve $T_i^{\frac{1}{p^{n}}}$ for all $1\leq i\leq d$ and all $n\geq 0$. By definition of $R_{\infty}$, we deduce $i\circ j$ is the identity morphism modulo $\xi$. This shows that $j$ is the inverse of $i$ as desired.
	\end{proof}
	\begin{rmk}[$\Gamma$-action on $\Ainf(R_{\infty})$]\label{Rmk-Perfection}
		During the proof of Lemma \ref{Lem-Perfection}, we have shown that
		\begin{equation}\label{Equ-Decomposition}
			\Ainf(R_{\infty})=\widehat \bigoplus_{\alpha\in (\bZ[1/p]\cap[0,1))^d}A^{\Box}_{\inf}(R)\underline T^{\alpha},
		\end{equation}
		where $\widehat \bigoplus$ means the topological direct sum with respect to the $(p,\xi)$-adic topology and for any $\alpha = (\alpha_1, \cdots, \alpha_d)\in (\bZ[1/p]\cap[0,1))^d, \underline T^{\alpha} = T_1^{\alpha_1}\cdots T_d^{\alpha_d}$.
		
		We know $\Spa(R_{\infty}[\frac{1}{p}],R_{\infty})$ is a Galois cover of $\Spa(R[\frac{1}{p}],R)$ with Galois group $\Gamma = \oplus_{i=1}^d\Zp\gamma_i$, where $\gamma_i(T_j^{\frac{1}{p^n}}) =\zeta_{p^n}^{\delta_{ij}}T_j^{\frac{1}{p^n}}$ for $1\leq i,j\leq d$. As $R_{\infty}$ is perfectoid, by deformation theory, the above $\Gamma$-action lifts uniquely to a continuous action on $\Ainf(R_{\infty})$ determined by $\gamma_i(T_j) = [\epsilon]^{\delta_{ij}}T_j$ for all $1\leq i,j\leq d$. From this, one can easily check that the equation (\ref{Equ-Decomposition}) is a topological decomposition of $\Gamma$-modules.
		In particular, $A^{\Box}_{\inf}(R)$ is equipped with the restricted $\Gamma$-action.
		
		Clearly, all above $\Gamma$-actions on $\Ainf(R_{\infty})$ and $A^{\Box}_{\inf}(R)$ commute with the actions of $\varphi$.
	\end{rmk}
	\begin{lem}\label{Lem-Special cover}
		The prism $(\Ainf(R_{\infty}),(\xi))\in (R/\Ainf)_{\Prism}$ is a cover of the final object of the topos ${\rm Shv}((R/\Ainf)_{\Prism})$. In particular, the prism $(A^{\square}_{\inf}(R),(\xi))$ is also a cover of the final object of the topos ${\rm Shv}((R/\Ainf)_{\Prism})$.
	\end{lem}
	\begin{proof}
		It can be checked directly that $R_{\infty}$ is a quasi-syntomic cover of $R$ in the sense of \cite[Definition 4.10(3)]{BMS-b}. So for any bounded prism $(A,I)\in (R/\Ainf)_{\Prism}$, $A/I\widehat \otimes_{R}R_{\infty}$ is also a quasi-syntomic cover of $A/I$. By \cite[Proposition 7.11]{BS}, there exists a prism $(B,IB)$ over $(A,I)$ (hence over $(\Ainf,(\xi))$) together with a $p$-faithfully flat morphism $A/I\widehat \otimes_RR_{\infty}\rightarrow B/IB$. In particular, $(B,IB)$ is a cover of $(A,I)$. Now, the composition $R_{\infty}\rightarrow A/I\widehat \otimes_RR_{\infty}\rightarrow B/IB$ lifts (uniquely) to a morphism $(\Ainf(R_{\infty}),(\xi))\rightarrow (B,IB)$ by deformation theory as $R_{\infty}$ is perfectoid. Since $(B,IB)$ is a cover of $(A,I)$, we complete the proof.
	\end{proof}
	
	Let $\widehat{A^{\Box}_{\inf}(R)[\frac{1}{\xi}]}$ be the $p$-adic completion of $A^{\Box}_{\inf}(R)[\frac{1}{\xi}]$. As $\varphi(\xi) = \xi^p+p\delta(\xi)$ with $p\delta(\xi)$ topologically nilpotent (since $\delta(\xi)\in \rA_{\inf}^{\times}$), the $\varphi$-action on $A^{\Box}_{\inf}(R)$ extends (uniquely) to $\widehat{A^{\Box}_{\inf}[\frac{1}{\xi}]}$. Similarly, if we denote by $\widehat{\Ainf(R_{\infty})[\frac{1}{\xi}]}$ the $p$-adic completion of $\Ainf(R_{\infty})[\frac{1}{\xi}]$, then it also admits an induced $\varphi$-action. Moreover, as $\Gamma$ fixes the base ring $\Ainf$, both $\widehat{A^{\Box}_{\inf}(R)[\frac{1}{\xi}]}$ and $\widehat{\Ainf(R_{\infty})[\frac{1}{\xi}]}$ are endowed with (continuous) $\Gamma$-actions as well.
	
	We recall the following proposition in \cite{Wu}.
	\begin{prop}[\emph{\cite[Theorem 4.6]{Wu}}]\label{Prop-Wu}
		Let $(A,I)$ be a bounded prism such that $\varphi(I)$ modulo $p$ is generated by a non-zero divisor in $A/p$. Let $(A_{\infty},IA_{\infty})$ be the perfection of the prism $(A,I)$, then we have an equivalence of categories
		\[\Et\Phi\rM(\widehat{A[\frac{1}{I}]})\rightarrow \Et\Phi\rM(\widehat{A_{\infty}[\frac{1}{I}]})\]
		induced by base change.
	\end{prop}
	
	\begin{rmk}\label{Rmk-Preserve extension}
		Keep notations as in Proposition \ref{Prop-Wu}. If moreover $I=(d)$ such that $(p,d)$ is regular in $A$, then the equivalence in Proposition \ref{Prop-Wu} preserves tensor products, dualities and extensions. The first two assertions are obvious. To see the third, one needs to check the proof of \cite[Theorem 4.6]{Wu} carefully. We split the sketch of the proof in steps.
		
		$(1)$~ Let $R$ be as in \cite[Proposition 4.2]{Wu}. Then the equivalence there preserves extensions. Denote $R_{\infty} = \colim_{\varphi} R$. It suffices to check short exact sequences of \'elate $\varphi$-modules over $R_{\infty}$ lift to short exact sequences of \'elate $\varphi$-modules over $R$. Since all modules and morphisms lift to $R_n = R$ for some $n$, this follows from the proof of \cite[Proposition 4.2]{Wu}.
		
		$(2)$~ Let $R$ be as in \cite[Proposition 4.3]{Wu}. Let $\widehat R_{\infty}$ be the $p$-adic completion of $R_{\infty}$. Assume $R$ is $p$-torsion free (and thus so are $R_{\infty}$ and $\widehat R_{\infty}$). Then the equivalence there preserves extensions. Since all rings involved are $p$-torsion free, this can be checked modulo $p^n$ and then reduces to the above case.
		
		$(3)$~ Under the assumption on $(A,(d))$, the equivalence in \cite[Theorem 4.6]{Wu} preserves extensions. By the argument in the proof of \cite[Theorem 4.4]{Wu}, it is enough to check that for a perfect $F_p$-algebra $R$, the equivalence between the category of \'etale $\varphi$-modules over $\rW(R)$ and the category of lisse $\Zp$-sheaves on $X = \Spec(R)$ preserves extensions. It suffices to check this modulo $p^n$. We regard lisse $\Zp/p^n$-sheaves as $\Zp/p^n$-representations of the algebraic fundamental group $G$ of $X$.
		For a short exact sequence of $\Zp/p^n$-representations
		\[\xymatrix@C=0.45cm{
			0\ar[r]& L_1\ar[r] &L_2\ar[r] &L_3\ar[r]& 0,
		}\]
		choose a finite \'etale Galois cover $Y = \Spec(S)$ of $X$ trivializing all $L_i$'s. Denote by $G_S$ the corresponding Galois group. Since the \'etale $\varphi$-modules corresponding to $L_i$ are $(\rW_n(S)\otimes L_i)^{G_S}$, the resulting sequence of \'etale $\varphi$-modules is exact by Hilbert's theorem $90$.
		Conversely, for any short exact sequence of \'etale $\varphi$-modules
		\[\xymatrix@C=0.45cm{
			0\ar[r]& M_1\ar[r] &M_2\ar[r] &M_3\ar[r]& 0,
		}\]
		denote $L_i$ the corresponding representations for all $i$ and choose $S$ as above. Then $L_i = (M_i\otimes_{\rW_n(R)}\rW_n(S))^{\varphi=1}$. Since taking $\varphi$-invariants is a left exact functor, we get the desired exactness by counting dimensions.
	\end{rmk}
	
	\begin{prop}\label{Prop-equivalence of phi-Gamma modules}
		$(1)$~ The functor $\Et\Phi\rM(\widehat{A^{\Box}_{\inf}(R)[\frac{1}{\xi}]})\rightarrow \Et\Phi\rM(\widehat{\Ainf(R_{\infty})[\frac{1}{\xi}]})$ induced by the base change $M\mapsto M\otimes_{\widehat{A^{\Box}_{\inf}(R)[\frac{1}{\xi}]}}\widehat{\Ainf(R_{\infty})[\frac{1}{\xi}]}$ is a tensor equivalence.
		
		$(2)$~ The functor $\Et\Phi\Gamma\rM(\widehat{A^{\Box}_{\inf}(R)[\frac{1}{\xi}]})\rightarrow \Et\Phi\Gamma\rM(\widehat{\Ainf(R_{\infty})[\frac{1}{\xi}]})$ induced by the base change $M\mapsto M\otimes_{\widehat{A^{\Box}_{\inf}(R)[\frac{1}{\xi}]}}\widehat{\Ainf(R_{\infty})[\frac{1}{\xi}]}$ is a tensor equivalence.
	\end{prop}
	\begin{proof}
		$(1)$~ This is a special case of Proposition \ref{Prop-Wu} (in which $(A,I)=(A^{\Box}_{\inf}(R),(\xi))$).
		
		$(2)$~ The fully faithful part follows from $(1)$. It is enough to show the essential surjectivity. For simplicity, put $A^{\Box} := \widehat{A^{\Box}_{\inf}(R)[\frac{1}{\xi}]}$ and $A_{\infty}:=\widehat{\Ainf(R_{\infty})[\frac{1}{\xi}]}$.
		
		Let $M_{\infty}\in \Et\Phi\Gamma(A_{\infty})$ and $M$ be the corresponding \'etale $\varphi$-module over $A^{\Box}$ via the equivalence in $(1)$. Then we have $\End_{\varphi}(M)\simeq \End_{\varphi}(M_{\infty})$. As $M_{\infty}$ is a $(\varphi,\Gamma)$-module, we get a continuous homomorphism $\Gamma\rightarrow\End_{\varphi}(M_{\infty})^{\times}\simeq \End_{\varphi}(M)^{\times}$. Therefore, $M$ is endowed with a continuous $\Gamma$-action commuting with $\varphi$ such that the isomorphism $M\otimes_{A^{\Box}}A_{\infty}\simeq M_{\infty}$ is $\Gamma$-equivariant. As a consequence, if we regard $M$ as an \'etale $(\varphi,\Gamma)$-module over $A^{\Box}$, then $M\otimes_{A^{\Box}}A_{\infty}$ is isomorphic to $M_{\infty}$ in $\Et\Phi\Gamma\rM(A_{\infty})$. This gives the essential surjectivity as desired.
		
		We have established the equivalence parts of both $(1)$ and $(2)$. It follows from some standard linear algebra-theoretic arguments that these are both tensor equivalences.
	\end{proof}
	
	Let $\Crys({\calO_{\Prism}[\frac{1}{\xi}]^{\wedge}_p},\varphi)$ be the category of prismatic $F$-crystals over ${\calO_{\Prism}[\frac{1}{\xi}]^{\wedge}_p}$ on $(R/\Ainf)_{\Prism}$. For any $\calM\in \Crys({\calO_{\Prism}[\frac{1}{\xi}]^{\wedge}_p},\varphi)$, we get an \'etale $\varphi$-module $\ev^{\Box}(\calM)$ by evaluating it at $(A^{\Box}_{\inf}(R)\rightarrow R=R)\in (R/\Ainf)_{\Prism}$. This gives a functor \[\ev^{\Box}: \Crys({\calO_{\Prism}[\frac{1}{\xi}]^{\wedge}_p},\varphi)\rightarrow \Et\Phi\rM(\widehat{A^{\Box}_{\inf}(R)[\frac{1}{\xi}]}).\]
	Similarly, the evaluation at $(A_{\inf}(R_{\infty})\rightarrow R_{\infty}\leftarrow R)\in (R/\Ainf)_{\Prism}$ gives a functor
	\[\ev_{\infty}: \Crys({\calO_{\Prism}[\frac{1}{\xi}]^{\wedge}_p},\varphi)\rightarrow \Et\Phi\rM(\widehat{\Ainf(R_{\infty})[\frac{1}{\xi}]}).\]
	Therefore, we get a commutative diagram as follows:
	\begin{equation}\label{Equ-commutative diagram of evulations-Phi}
		\xymatrix@C=0.5cm{
			& \ar[ld]_{\ev^{\Box}}\Crys({\calO_{\Prism}[\frac{1}{\xi}]^{\wedge}_p},\varphi)\ar[rd]^{\ev_{\infty}}&\\
			\Et\Phi\rM(\widehat{A^{\Box}_{\inf}(R)[\frac{1}{\xi}]})\ar[rr]_{-\otimes_{\widehat{A^{\Box}_{\inf}(R)[\frac{1}{\xi}]}}\widehat{\Ainf(R_{\infty})[\frac{1}{\xi}]}}&&\Et\Phi\rM(\widehat{\Ainf(R_{\infty})[\frac{1}{\xi}]})
			.}
	\end{equation}
	
	
	\begin{lem}\label{Lem-The first evaluation}
		The evaluation functor $\ev^{\Box}$ can be upgraded to a functor from the category $\Crys({\calO_{\Prism}[\frac{1}{\xi}]^{\wedge}_p},\varphi)$ to the category $\Et\Phi\Gamma\rM(\widehat{A^{\Box}_{\inf}(R)[\frac{1}{\xi}]})$.
	\end{lem}
	\begin{proof}
		For any given $\calM\in \Crys({\calO_{\Prism}[\frac{1}{\xi}]^{\wedge}_p},\varphi)$, we only need to specify the $\Gamma$-action on $\ev^{\Box}(\calM)$. For simplicity, we denote $\widehat{A^{\Box}_{\inf}(R)[\frac{1}{\xi}]}$ by $A^{\Box}$.
		
		Let $M:=\ev^{\Box}(\calM)$. As we have seen in Remark \ref{Rmk-Perfection}, $A^{\Box}_{\inf}(R)$ admits a continuous $\Gamma$-action and hence one can regard elements in $\Gamma$ as automorphisms of the prism $(A^{\Box}_{\inf}(R),(\xi))$. As $\calM$ is a crystal, for any $\gamma\in \Gamma$, we get an isomorphism of $\varphi$-modules$\gamma^{\ast}M = M\otimes_{A^{\Box},\gamma}A^{\Box}\rightarrow M$ over $A^{\Box}$.
		In other words, we get a $\gamma$-action on $M$, which commutes with $\varphi$. As this holds for any $\gamma\in \Gamma$, we see $M$ is an \'etale $(\varphi,\Gamma)$-module as desired.
	\end{proof}
	\begin{rmk}\label{Rmk-The first evaluation}
		By the same argument, we see that $\ev_{\infty}$ can be upgraded to a functor from $\Crys({\calO_{\Prism}[\frac{1}{\xi}]^{\wedge}_p},\varphi)$ to  $\Et\Phi\Gamma\rM(\widehat{\Ainf(R_{\infty})[\frac{1}{\xi}]})$. By construction, the $\Gamma$-action on ${\ev_{\infty}}(\calM)$ induced by $-\otimes_{\widehat{A^{\Box}_{\inf}(R)[\frac{1}{\xi}]}}\widehat{\Ainf(R_{\infty})[\frac{1}{\xi}]}$ is compatible with the action induced by viewing $\Gamma$ as automorphisms of the prism $(\Ainf(R_{\infty}),(\xi))$ and taking pull-backs as in the proof of Lemma \ref{Lem-The first evaluation}. We still denote the resulting functor
		\[\Crys({\calO_{\Prism}[\frac{1}{\xi}]^{\wedge}_p},\varphi)\rightarrow \Et\Phi\Gamma\rM(\widehat{\Ainf(R_{\infty})[\frac{1}{\xi}]})\]
		by $\ev^{\Box}$ when the contexts are clear for the sake that we shall give another description of $\Gamma$-actions on evaluations of $\ev_{\infty}$ in a moment.
	\end{rmk}
	\begin{rmk}
		We will see $\ev^{\Box}$ is indeed an equivalence (See Proposition \ref{Prop-The second evaluation} and Theorem \ref{Thm-F crystals as Phi-Gamma modules}).
	\end{rmk}
	
	Now, we are going to study $\ev_{\infty}$. As $(A_{\inf}(R_{\infty})\rightarrow R_{\infty}\leftarrow R)$ is  a cover of the final object of the topos  ${\rm Shv}(R/\Ainf)_{\Prism}$, the evaulation of an $F$-crystal on the prism $(A_{\inf}(R_{\infty})\rightarrow R_{\infty}\leftarrow R)$ gives an object in $\Et\Phi\rM(\Ainf(R_{\infty}))$ with stratifications as what Morrow-Tsuji did in \cite[$\S$ 3.1]{MT}. More precisely, let $(A^{\bullet}_{\infty},(\xi))$ be the \v{C}ech nerve of $(\Ainf(R_{\infty}),(\xi))$ in $(R/\Ainf)_{\Prism}$. Denote $p_0,p_1: A_{\infty}^0\rightarrow A_{\infty}^1$ and $p_{01},p_{02},p_{12}: A_{\infty}^1\rightarrow A_{\infty}^2$ the face maps in low degrees and denote $\Delta:A_{\infty}^1\rightarrow A_{\infty}^0$ the degeneracy map in degree $1$. Then we make the following definition as an analogue of \cite[Definition 3.14]{MT}.
	\begin{dfn}\label{Dfn-Stratification}
		An \emph{stratification} on an \'etale $\varphi$-module $M$ over $\widehat{A_{\infty}^0[\frac{1}{\xi}]} = \widehat{\Ainf(R_{\infty})[\frac{1}{\xi}]}$ with respect to $\widehat{A_{\infty}^{\bullet}[\frac{1}{\xi}]}$ is an isomorphism
		\[\varepsilon: M\otimes_{\widehat{A_{\infty}^0[\frac{1}{\xi}]},p_1}\widehat{A_{\infty}^1[\frac{1}{\xi}]}\rightarrow M\otimes_{\widehat{A_{\infty}^0[\frac{1}{\xi}]},p_0}\widehat{A_{\infty}^1[\frac{1}{\xi}]}\]
		of $\varphi$-modules over $\widehat{A_{\infty}^1[\frac{1}{\xi}]}$ satisfying the cocycle condition:
		\[p_{01}^{\ast}(\varepsilon)\circ p_{12}^{\ast}(\varepsilon) = p_{02}^{\ast}(\varepsilon): M\otimes_{\widehat{A_{\infty}^0[\frac{1}{\xi}]},q_2}\widehat{A_{\infty}^2[\frac{1}{\xi}]}\rightarrow M\otimes_{\widehat{A_{\infty}^0[\frac{1}{\xi}]},q_0}\widehat{A_{\infty}^2[\frac{1}{\xi}]}.\]
		
		We denote ${\rm Strat}(\widehat{A_{\infty}^{\bullet}[\frac{1}{\xi}]})$ the category of \'etale $\varphi$-modules over $\widehat{A_{\infty}^0[\frac{1}{\xi}]}$ with stratifications with respect to $\widehat{A_{\infty}^{\bullet}[\frac{1}{\xi}]}$.
	\end{dfn}
	
	\begin{lem}\label{Lem-Fcrystal as Stratification}
		The evaluation functor $\ev_{\infty}$ induces an equivalence between the category $\Crys({\calO_{\Prism}[\frac{1}{\xi}]^{\wedge}_p},\varphi)$  and the category ${\rm Strat}(\widehat{A_{\infty}^{\bullet}[\frac{1}{\xi}]})$.
	\end{lem}
	\begin{proof}
		This follows from \cite[Proposition 3.2]{Wu}.
	\end{proof}
	
	Although $(A_{\infty}^0,(\xi))$ is a perfect prism, for any $i\geq 1$, $(A_{\infty}^i,(\xi))$ is not perfect. We will replace the cosimplicial object $(A_{\infty}^{\bullet},(\xi))$ by another one $B_{\infty}^{\bullet}$, of which each term is perfect. The $(B_{\infty}^{\bullet},(\xi))$ is constructed as follows.
	
	By the same argument as in the proof of Lemma \ref{Lem-Special cover}, we see $(\Ainf(R_{\infty}),(\xi))$ is also a cover of the final object in ${\rm Shv}((R/\Ainf)_{\Prism}^{\perf})$. Let $(B_{\infty}^{\bullet},(\xi))$ be the \v{C}ech nerve of $(\Ainf(R_{\infty}),(\xi))$ in $(R/\Ainf)^{\perf}_{\Prism}$.
	
	\begin{lem}\label{Lem-Perfect special cover}
		For any $i\geq 0$, $(B_{\infty}^i,(\xi))$ is the perfection of $(A_{\infty}^i,(\xi))$.
	\end{lem}
	\begin{proof}
		Clearly, $A_{\infty}^0 = \Ainf(R_{\infty}) = B^0_{\infty}$. For $i\geq 1$, by definition of $(A_{\infty}^i,(\xi))$, it is the initial object among the category of prisms $(A,I)\in (R/\Ainf)_{\Prism}$ together with $(i+1)$-morphisms from $(A_{\infty}^0,(\xi))$ to $(A,I)$. Similarly, $(B_{\infty}^i,(\xi))$ is the initial object among the category of perfect prisms $(B,J)\in (R/\Ainf)_{\Prism}^{\perf}$ together with $(i+1)$-morphisms from $(A_{\infty}^0,(\xi))$ to $(B,J)$. In particular, we get a unique morphism $(A_{\infty}^i,(\xi))\rightarrow (B_{\infty}^i,(\xi))$ of prisms, which factors uniquely through the perfection $(A_{\infty}^{i},(\xi))^{\perf}$ of the source. Denote this morphism by $f: (A_{\infty}^{i},(\xi))^{\perf}\rightarrow (B_{\infty}^i,(\xi))$. On the other hand, the $(i+1)$ arrows from $(A_{\infty}^0,(\xi))$ to $(A_{\infty}^i,(\xi))$ composed with the canonical morphism $(A_{\infty}^i,(\xi))\rightarrow (A_{\infty}^i,(\xi))^{\perf}$ induces a unique morphism $g:(B_{\infty}^i,(\xi))\rightarrow (A_{\infty}^{i},(\xi))^{\perf}$. By construction, both $f\circ g$ and $g\circ f$ are the identity morphisms on their domains, which completes the proof.
	\end{proof}
	
	Similar to Definition \ref{Dfn-Stratification}, one can define stratifications on \'etale $\varphi$-modules in $\Et\Phi\rM(\widehat{A_{\infty}^0[\frac{1}{\xi}]})$ with respect to $\widehat{B_{\infty}^{\bullet}[\frac{1}{\xi}]}$ and we denote ${\rm Strat}(\widehat{B_{\infty}^{\bullet}[\frac{1}{\xi}]})$ the corresponding category. Then we have the following lemma.
	
	\begin{lem}\label{Lem-Fcrystal as perfect Stratification}
		Let $M\in \Et\Phi\rM(\widehat{A_{\infty}^0[\frac{1}{\xi}]})$. Then $M$ admits a stratification with respect to $\widehat{A_{\infty}^{\bullet}[\frac{1}{\xi}]}$ if and only if it admits a stratification with respect to $\widehat{B_{\infty}^{\bullet}[\frac{1}{\xi}]}$.
		
		As a consequence, the evaluation functor $\ev_{\infty}$ induces an equivalence between the category $\Crys({\calO_{\Prism}[\frac{1}{\xi}]^{\wedge}_p},\varphi)$ and the category ${\rm Strat}(\widehat{B_{\infty}^{\bullet}[\frac{1}{\xi}]})$.
	\end{lem}
	\begin{proof}
		As all morphisms $\varepsilon$, $\Delta^{\ast}(\varepsilon)$, $p_{\bullet}^{\ast}(\varepsilon)$ for $\bullet\in \{01,02,12\}$ involved in the definition of stratifications (Definition \ref{Dfn-Stratification}) are isomorphisms of \'etale $\varphi$-modules. The lemma follows from Proposition \ref{Prop-Wu} and Lemma \ref{Lem-Fcrystal as Stratification} immediately.
	\end{proof}
	
	\begin{rmk}\label{Rmk-Realtion with Bhatt-Scholze}
		As $(\Ainf(R_{\infty}),(\xi))$ is a perfect prism, by deformation theory, to give $(i+1)$ morphisms from $(A_{\infty}^0,(\xi))$ to $(A,(\xi))$ in $(R/\Ainf)_{\Prism}$ is equivalently to give $(i+1)$ $R$-morphisms from $R_{\infty}$ to $A/I$, which is also equivalent to give an $R$-morphism $S^i\rightarrow A/I$ where $S^i = R_{\infty}\widehat \otimes_R\cdots\widehat \otimes_RR_{\infty}$ is the $(i+1)$-folds $p$-complete tensor product of $R_{\infty}$ over $R$. It is easy to check that $S_i$ is a quasi-regular semiperfectoid ring in the sense of \cite[Notation 7.1]{BS}. It follows from \cite[Proposition 7.2]{BS} that $(A_{\infty}^i,(\xi)) = (\Prism_{S_i}^{\mathrm{init}},(\xi))$.
	\end{rmk}
	
	It remains to consider stratifications with respect to $\widehat{B_{\infty}^{\bullet}[\frac{1}{\xi}]}$ on \'etale $\varphi$-modules over $\widehat{\Ainf(R_{\infty})[\frac{1}{\xi}]}$. To do so, we need to describe the \v{C}ech nerve $\widehat{B_{\infty}^{\bullet}[\frac{1}{\xi}]}$ in a more explicit way.
	
	\begin{lem}\label{Lem-Describe of perfect cover}
		$(1)$~ For $i = 0$, we have $\widehat{B_{\infty}^{0}[\frac{1}{\xi}]} = \rW((R_{\infty}[\frac{1}{p}])^{\flat})$. The ring $\Ainf(R_{\infty})$ with $(p,\xi)$-adic topology is an open bounded subring of the both sides.
		
		$(2)$~ In general, for any $i\geq 1$, $\widehat{B_{\infty}^{i}[\frac{1}{\xi}]} = \rC(\Gamma^i,\rW((R_{\infty}[\frac{1}{p}])^{\flat})$ is the ring of continuous functions on $\Gamma^i$. Moreover, $\widehat{B_{\infty}^{\bullet}[\frac{1}{\xi}]}\simeq \rC(\Gamma^{\bullet},\rW((R_{\infty}[\frac{1}{p}])^{\flat})$ is an isomorphism of cosimplicial rings.
	\end{lem}
	\begin{proof}
		$(1)$~ Since both sides are $p$-complete $\Zp$-algebras on which $\varphi$ acts as automorphisms, it is enough to verify the equality modulo $p$, which holds obviously.
		
		$(2)$~ We follow the idea of the proof of \cite[Lemma 5.3]{Wu}. Let $S_{\infty}^i = B_{\infty}^i/\xi$, which is a perfectoid algebra, for any $i\geq 0$. We claim that for any $i\geq 0$,
		\begin{equation}\label{Equ-Key point in Perfect cover}
			S_{\infty}^i[\frac{1}{p}] = \rC(\Gamma^i,R_{\infty})[\frac{1}{p}].
		\end{equation}
		
		By \cite[Theorem 3.10]{BS}, it follows from the construction of $B_{\infty}^{\bullet}$ that $S_{\infty}^{\bullet}$ is the \v{C}ech nerve associated to $R_{\infty}$ in the category of perfectoid $R$-algebras and $S_{\infty}^i$ is the initial object in the category of perfectoid $R$-algebras as the target of $(i+1)$ morphisms from $R_{\infty}$. As a consequence, $S_{\infty}^i[\frac{1}{p}]$ is the initial object in the category of perfectoid $R[\frac{1}{p}]$-algebras as the target of $(i+1)$ morphisms from $R_{\infty}$.
		
		Let $S_{\infty}^i[\frac{1}{p}]^+$ be the $p$-adic completion of the integral closure of $S_{\infty}^i$ in $S_{\infty}^i[\frac{1}{p}]$. Then it is almost isomorphic to the quotient of $S_{\infty}^i$ by its $p$-torsion part, which is again a perfectoid algebra for all $i\geq 0$. More precisely, the natural map
		\[S_{\infty}^i\to S_{\infty}^i[\frac{1}{p}]^+\]
		has kernel and cokernel which are both killed by $\frakm_{\Cp}$.
		From this, we see that $V_i: = \Spa(S_{\infty}^i[\frac{1}{p}],S_{\infty}^i[\frac{1}{p}]^+)$ is the final object in the category of affnoid perfectoid spaces over $U = \Spa(R[\frac{1}{p}],R)$ together with $(i+1)$ morphisms to $V_0 = \Spa(R_{\infty}[\frac{1}{p}],R_{\infty})$.
		
		We regard $U$ and $V_i$ for $i\geq 0$ as diamonds. Then the above argument shows that $V_i^{\diamond}$ is the $(i+1)$-folds fibre products of $V_0^{\diamond}$ over $U^{\diamond}$. As $V_1$ is a Galois pro-\'etale cover of $U$ with the Galois group $\Gamma$, we know that $V_1^{\diamond}/\Gamma = U^{\diamond}$ and that $V_1^{\diamond}\rightarrow U^{\diamond}$ is a $\Gamma$-torsor. Therefore, we deduce that $V_i^{\diamond}\simeq V_1^{\diamond}\times\underline \Gamma^i$. Since all spaces involved are defined over $\Spa(\Cp,\calO_{\Cp})$, by \cite[Example 11.12]{Sch-a}, we see that $V_1^{\diamond}\times\underline \Gamma^i\simeq V_1^{\diamond}\times_{\Spa(\Cp,\calO_{\Cp})^{\diamond}}\Spa(\rC(\Gamma^i,\Cp),\rC(\Gamma^i,\calO_{\Cp}))^{\diamond}$.
		By \cite[Proposition 10.2.3]{SW}, we have
		\begin{equation*}
			\begin{split}
				& V_1^{\diamond}\times_{\Spa(\Cp,\calO_{\Cp})^{\diamond}}\Spa(\rC(\Gamma^i,\Cp),\rC(\Gamma^i,\calO_{\Cp}))^{\diamond}\\
				\simeq ~& ~\Spa(R_{\infty}\widehat \otimes_{\calO_{\Cp}}\rC(\Gamma^i,\calO_{\Cp})[\frac{1}{p}],R_{\infty}\widehat \otimes_{\calO_{\Cp}}\rC(\Gamma^i,\calO_{\Cp}))^{\diamond}\\
				\simeq~ & ~\Spa(\rC(\Gamma^i,R_{\infty})[\frac{1}{p}],\rC(\Gamma^i,R_{\infty}))^{\diamond}.
			\end{split}
		\end{equation*}
		Combining all isomorphisms above, we deduce that
		\[V_i^{\diamond}\simeq \Spa(\rC(\Gamma^i,R_{\infty})[\frac{1}{p}],\rC(\Gamma^i,R_{\infty}))^{\diamond}.\]
		By \cite[Proposition 10.2.3]{SW} again, we have $S_{\infty}^i[\frac{1}{p}] = \rC(\Gamma^i,R_{\infty})[\frac{1}{p}]$ for all $i\geq 0$ as desired.
		
		Similar to the proof of $(1)$, in order to prove $(2)$, we are reduced to verify
		\[\widehat{B_{\infty}^{\bullet}[\frac{1}{\xi}]}/p = B_{\infty}^{\bullet}[\frac{1}{\xi}]/p\simeq \rC(\Gamma^i,(R_{\infty}[\frac{1}{p}])^{\flat}).\]
		Combing the claim $(\ref{Equ-Key point in Perfect cover})$ with the tilting equivalence, we get $(S_{\infty}^i[\frac{1}{p}]^+)^{\flat}\simeq \rC(\Gamma^i,R_{\infty}^{\flat})$. On the other hand, since the map $S_{\infty}^i\twoheadrightarrow S_{\infty}^i[\frac{1}{p}]^+$ has kernel and cokernel which are annihilated by $\frakm_{\Cp}$, the composition $(S_{\infty}^i)^{\flat}\rightarrow (S_{\infty}^i[\frac{1}{p}]^+)^{\flat}\simeq \rC(\Gamma^i,R_{\infty}^{\flat})$ has kernel and cokernel which are annihilated by $\frakm_{\calO_{\Cp}^{\flat}}$, which implies that
		\[B_{\infty}^{\bullet}[\frac{1}{\xi}]/p\simeq (S_{\infty}^i[\frac{1}{p}])^{\flat}\simeq \rC(\Gamma^i,(R_{\infty}[\frac{1}{p}])^{\flat}).\]
		
		Finally, by chasing isomorphisms above, we see that \[\widehat{B_{\infty}^{\bullet}[\frac{1}{\xi}]}\simeq \rC(\Gamma^{\bullet},\rW((R_{\infty}[\frac{1}{p}])^{\flat})\]
		is an isomorphism of cosimplicial rings and complete the proof.
	\end{proof}
	
	\begin{exam}\label{Exam-low degree of cover}
		We identify $\widehat{B_{\infty}^{i}[\frac{1}{\xi}]}$ with $\rC(\Gamma^i,\rW((R_{\infty}[\frac{1}{p}])^{\flat})$ for $i\in \{0,1,2\}$ via the isomorphisms in Lemma \ref{Lem-Perfect special cover} $(2)$. Then
		\[p_j: \rW((R_{\infty}[\frac{1}{p}])^{\flat}) = \rC(\Gamma^0,\rW((R_{\infty}[\frac{1}{p}])^{\flat})\rightarrow \rC(\Gamma^1,\rW((R_{\infty}[\frac{1}{p}])^{\flat})\]
		for $j\in \{0,1\}$ is given by
		\begin{equation*}
			p_j(x) = \left\{
			\begin{array}{rcl}
				x, & \IF ~j=0\\
				\gamma(x), & \IF ~j=1
			\end{array}   .
			\right.
		\end{equation*}
		for any $x\in \rW((R_{\infty}[\frac{1}{p}])^{\flat})$ and $\gamma\in \Gamma$.
		The degeneracy morphism
		\[\Delta:\rC(\Gamma^1,\rW((R_{\infty}[\frac{1}{p}])^{\flat})\rightarrow\rW((R_{\infty}[\frac{1}{p}])^{\flat})\]
		is given by $\Delta(f) = f(1)$, for any continuous function $f:\Gamma\rightarrow \rW((R_{\infty}[\frac{1}{p}])^{\flat})$.
		For $j\in \{01,02,12\}$, for any $f\in \rC(\Gamma^1,\rW((R_{\infty}[\frac{1}{p}])^{\flat})$ and $\gamma_0,\gamma_1\in \Gamma$,
		\begin{equation*}
			p_j(f)(\gamma_0,\gamma_1) = \left\{
			\begin{array}{rcl}
				f(\gamma_0), & \IF ~j=01\\
				f(\gamma_0\gamma_1), & \IF ~j=02\\
				\gamma_0(f(\gamma_1)), & \IF ~j=12
			\end{array}.
			\right.
		\end{equation*}
	\end{exam}
	As an application of Lemma \ref{Lem-Describe of perfect cover}, we give the description of $\ev_{\infty}$.
	
	\begin{prop}\label{Prop-The second evaluation}
		The functor $\ev_{\infty}$ induces an equivalence of categories \[\Crys({\calO_{\Prism}[\frac{1}{\xi}]^{\wedge}_p},\varphi)\xrightarrow{\ev_{\infty}} \Et\Phi\Gamma\rM(\widehat{\Ainf(R_{\infty})[\frac{1}{\xi}]}).\]
	\end{prop}
	\begin{proof}
		It follows from Lemma \ref{Lem-Describe of perfect cover} combined with the Galois descent that the evaluation functor $\ev_{\infty}$ can be upgraded to a functor to  $\Et\Phi\Gamma\rM(\widehat{\Ainf(R_{\infty})[\frac{1}{\xi}]})$. By Lemma \ref{Lem-Fcrystal as perfect Stratification}, $\ev_{\infty}$ is an equivalence.
	\end{proof}
	
	\begin{rmk}\label{Rmk-Description of Gamma-action}
		For any $i\geq 0$, denote $C^i:=\rC(\Gamma^i,(R_{\infty}[\frac{1}{p}])^{\flat})\simeq \widehat{B_{\infty}^i[\frac{1}{\xi}]}$.
		Let $M\in\Et\Phi\rM(\widehat{\Ainf(R_{\infty})[\frac{1}{\xi}]})$ with a stratification $\varepsilon: M\otimes_{C^0,p_1}C_1\rightarrow M\otimes_{C^0,p_0}C_1$ with respect to $C^{\bullet}$. We describe the induced $\Gamma$-action on $M$ explicitly as follows.
		
		For any $\gamma\in \Gamma$ and $m\in M$, define $\gamma(m):=\varepsilon(m\otimes_{p_1} 1)(\gamma)$; that is, if $\varepsilon(m\otimes_{p_1} 1) = \sum_{i=1}^dm_i\otimes_{p_0} f_i$ for $m_i\in M$ and $f_i\in C^1$, then $\gamma(m)=\sum_{i=1}^df_i(\gamma)m_i$. One can easily check this ``action'' is semi-linear. As $\Delta^{\ast}(\varepsilon)=\id_M$ is the identity morphism on $M$, we see $m = \sum_{i=1}^df_i(1)m_i$, which shows that $1\in \Gamma$ acts identically on $M$. Finally, for any $\gamma_0,\gamma_1\in \Gamma$, one can check that $p_{02}^{\ast}(\varepsilon)(m\otimes_{q_2}1)(\gamma_0,\gamma_1) = (\gamma_0\gamma_1)(m)$ and that $(p_{01}^{\ast}(\varepsilon)\circ p_{12}^{\ast}(\varepsilon))(\gamma_0,\gamma_1)=\gamma_0(\gamma_1(m))$. So we get a $\Gamma$-action on $M$.
	\end{rmk}
	\begin{rmk}\label{Rmk-Compare Gamma-action with Morrow-Tsuji}
		We have seen $\Gamma$ acts on $A_{\infty}^0=\Ainf(R_{\infty})$ as automorphisms. By the construction of $A_{\infty}^{\bullet}$, for any $i\geq 0$, $A_{\infty}^i$ is equipped with an action of $\Gamma^{i+1}$ such that the face and degeneracy morphisms are compatible with these actions. These actions extend to $\widehat{B_{\infty}^{\bullet}[\frac{1}{\xi}]}$ and thus for any $M\in\Et\Phi\rM(\widehat{A_{\infty}^{0}[\frac{1}{\xi}]})$ with a stratification, one can define a $\Gamma$-action on $M$ as what Morrow-Tsuji did in the paragraph below \cite[Remark 3.15]{MT}. One can check this $\Gamma$-action coincides with the one given in Proposition \ref{Prop-The second evaluation}.
	\end{rmk}
	
	Combining Proposition \ref{Prop-The second evaluation} with Lemma \ref{Lem-The first evaluation}, we get a diagram
	\begin{equation}\label{Equ-commutative diagram of evulations-PhiGamma}
		\xymatrix@C=0.5cm{
			& \ar[ld]_{\ev^{\Box}}\Crys({\calO_{\Prism}[\frac{1}{\xi}]^{\wedge}_p},\varphi)\ar[rd]^{\ev_{\infty}}&\\
			\Et\Phi\Gamma\rM(\widehat{A^{\Box}_{\inf}(R)[\frac{1}{\xi}]})\ar[rr]_{-\otimes_{\widehat{A^{\Box}_{\inf}(R)[\frac{1}{\xi}]}}\widehat{\Ainf(R_{\infty})[\frac{1}{\xi}]}}&&\Et\Phi\Gamma\rM(\widehat{\Ainf(R_{\infty})[\frac{1}{\xi}]})
			,}
	\end{equation}
	which is commutative after forgetting $\Gamma$-actions. The main theorem says that this diagram is indeed commutative.
	
	\begin{thm}\label{Thm-F crystals as Phi-Gamma modules}
		The above diagram $(\ref{Equ-commutative diagram of evulations-PhiGamma})$ is commutative such that all arrows involved are equivalences of categories.
	\end{thm}
	\begin{proof}
		By Remark \ref{Rmk-The first evaluation}, to check the commutativity, it suffices to see the two functors
		\[\ev^{\Box},\ev_{\infty}:\Crys({\calO_{\Prism}[\frac{1}{\xi}]^{\wedge}_p},\varphi))\rightarrow \Et\Phi\Gamma\rM(\widehat{\Ainf(R_{\infty})[\frac{1}{\xi}]})\]
		coincide.
		
		For any given $\calM\in\Crys({\calO_{\Prism}[\frac{1}{\xi}]^{\wedge}_p},\varphi)$, let $\ev_{C^i}(\calM)$ be the evaluation of $\calM$ at $(B_{\infty}^i,(\xi))$ as \'etale $\varphi$-modules in $\Et\Phi\rM(\rC(\Gamma^i,\rW((R_{\infty}[\frac{1}{p}])^{\flat})))$. For any $\gamma\in \Gamma$, we claim the following diagram
		\begin{equation}\label{Equ-F-crystal as Phi-Gamma modules}
			\xymatrix@C=0.4cm{
				p_1^{\ast}\gamma^{\ast}\ev_{C^0}(\calM)\ar@{=}[r]\ar[d]& (1,\gamma)^{\ast}p_1^{\ast}\ev_{C^0}(\calM)\ar[r]&(1,\gamma)^{\ast}\ev_{C^1}(\calM)\ar[d]&
				\ar[l](1,\gamma)^{\ast}p_0^{\ast}\ev_{C^0}(\calM)\ar@{=}[d]\\
				p_1^{\ast}\ev_{C^0}(\calM)\ar[rr]&& \ev_{C^1}(\calM)&\ar[l]p_0^{\ast}\ev_{C^0}(\calM)
			}
		\end{equation}
		commutes with all arrows being isomorphisms of \'etale $\varphi$-modules over $\rC(\Gamma,\rW((R_{\infty}[\frac{1}{p}])^{\flat}))$, where $(1,\gamma)$ acts on $\widehat{B_{\infty}^1[\frac{1}{\xi}]}\simeq\rC(\Gamma,\rW((R_{\infty}[\frac{1}{p}])^{\flat}))$ as explained in Remark \ref{Rmk-Compare Gamma-action with Morrow-Tsuji} and for any continuous function $f:\Gamma\rightarrow \rC(\Gamma,\rW((R_{\infty}[\frac{1}{p}])^{\flat}))$, $(1,\gamma)(f)(\bullet) = f(\bullet\gamma)$.
		
		In fact, by Remark \ref{Rmk-Compare Gamma-action with Morrow-Tsuji}, we have the following commutative diagram of prisms
		\[\xymatrix@C=0.45cm{
			(B_{\infty}^0,(\xi))\ar[r]^{\gamma}\ar@{.>}[d]\ar[d]_{p_1}&(B_{\infty}^0,(\xi))\ar[d]^{p_1}\\
			(B_{\infty}^1,(\xi))\ar[r]_{(1,\gamma)}&(B_{\infty}^1,(\xi))
		}.\]
		Considering the evaluations of $\calM$ along this diagram, we see the left square of \ref{Equ-F-crystal as Phi-Gamma modules} is commutative. Similarly, one can prove the right square commutes by considering the diagram
		\[\xymatrix@C=0.45cm{
			(B_{\infty}^0,(\xi))\ar[r]^{1}\ar@{.>}[d]\ar[d]_{p_0}&(B_{\infty}^0,(\xi))\ar[d]^{p_0}\\
			(B_{\infty}^1,(\xi))\ar[r]_{(1,\gamma)}&(B_{\infty}^1,(\xi)).
		}.\]
		
		We explain how the claim implies the compatibility of $\ev^{\Box}$ and $\ev_{\infty}$. In fact, for any $m\in \ev_{\infty}(\calM)$ (just as a $\varphi$-module and so is equal to $\ev^{\Box}(\calM)$), consider $m\otimes_{\gamma}1\otimes_{p_1}1\in p_1^{\ast}\gamma^{\ast}\ev_{C^0}(\calM)$, its image in $p_0^{\ast}\ev_{C^0}(\calM)$ via the left-bottom arrows is $\varepsilon(\gamma(m)\otimes_{p_1}1)$ with $\gamma$-action being induced via $\ev^{\Box}$ and the image via the top-right arrows in $\varepsilon(m\otimes_{p_1}1)\otimes_{(1,\gamma)}1$. In other words, we have
		\[\varepsilon(\gamma(m)\otimes_{p_1}1)=\varepsilon(m\otimes_{p_1}1)\otimes_{(1,\gamma)}1.\]
		Assume $\varepsilon(m\otimes_{p_1}1)=\sum_{i=1}^dm_i\otimes_{p_0}f_i$ for $m_i\in M$ and $f_i\in \rC(\Gamma,\rW((R_{\infty}[\frac{1}{p}])^{\flat}))$. Evaluating both sides at $1\in \Gamma$, we get
		\[\gamma(m) = \sum_{i=1}^df_i(\gamma)m_i.\]
		According to Remark \ref{Rmk-Description of Gamma-action}, the right hand side is the $\gamma$-action on $m$, which shows the compatibility as desired.

		The last assertion follows from Proposition \ref{Prop-equivalence of phi-Gamma modules} $(2)$, Proposition \ref{Prop-The second evaluation} and the commutativity of the diagram $(\ref{Equ-commutative diagram of evulations-PhiGamma})$.
	\end{proof}
	\begin{rmk}\label{Rmk-General groups}
		One can replace $R_{\infty}$ by $\widehat{\overline R}$ the $p$-adic completion of the normalization of $R$ in the filtered colimit of all finite \'etale Galois extensions of $R$ which are unramified outside $p$, that is, $\Spa(\widehat{\overline R}[\frac{1}{p}],\widehat{\overline R})$ is the ``universal cover'' of $\Spa(R[\frac{1}{p}],R)$ and replace $\Gamma$ by $G = \pi_1^{\et}(\Spa(R[\frac{1}{p}],R))$ (\cite[Example 5.5 (ii)]{MT}). It is easy to see $\widehat{\overline R}$ is a perfectoid $R$-alegbra and is also a quasi-syntomic cover of $R$. Let $\widetilde A=\widehat{\Ainf(\widehat{\overline R})[\frac{1}{\xi}]}$. Then, by the same argument as above, one can show that the evaluation at $(\widetilde A,(\xi))$ induces an equivalence from the category $\Crys({\calO_{\Prism}[\frac{1}{\xi}]^{\wedge}_p},\varphi)$ to the category of \'etale $(\varphi,G)$-modules over $\widetilde A$.
	\end{rmk}
	
	Until now we have assumed $R$ lives over $\calO$. By similar arguments as above, we have the following corollary.
	\begin{cor}\label{descend equivalence}
		Let $\Spf(R)$ be a $p$-adic smooth formal scheme over $\calO_K$ with a framing $\square:\calO_K\langle\underline{T}^{\pm1}\rangle\to R$, where $\calO_K$ is the ring of integers in a $p$-adic field $K$. Then there is an equivalence between the category of $F$-crystals over $\calO_{\Prism}[\frac{1}{I}]^{\wedge}_p$ on the absolute prismatic site $(R)_{\Prism}$ and the category of \'etale $(\varphi, G)$-modules over $\widehat{\Ainf(\widehat{\overline R})[\frac{1}{\xi}]}$, where $G = \pi_1^{\et}(\Spa(R[\frac{1}{p}],R))$.
	\end{cor}

	\section{Prismatic crystals and relative $(\varphi, \Gamma)$-modules} \label{correspondence}
	
	Let $\frakX$ be a separated $p$-adic smooth formal scheme over $\calO_K$ and $X$ be its adic generic fiber. Let ${\rm Crys}(\calO_{\Prism}[\frac{1}{I}]^{\wedge}_p, \varphi)$ denote the category of $F$-crystals over $\calO_{\Prism}[\frac{1}{I}]^{\wedge}_p$ on the absolute prismatic site $(\frakX)_{\Prism}$.

	The main result of this section is the following theorem.
	\begin{thm}\label{local system}
		There is an equivalence of categories
		\[
		{\rm Crys}(\calO_{\Prism}[\frac{1}{I}]^{\wedge}_p, \varphi)\xrightarrow{\simeq} {\rm LS}(X_{\rm \acute et}, \bZ_p),
		\]
		where the right one is the category of \'etale $\bZ_p$-local systems on $X$.
	\end{thm}
	
	\begin{proof}
		By \cite[Theorem 9.3.7]{KL}, there is a natural equivalence between the category $\Et\Phi(W(\widehat{\calO}_{X^{\flat}}))$ of \'etale $\varphi$-modules over $W(\widehat{\calO}_{X^{\flat}})$ and the category ${\rm LS}(X_{\rm \acute et}, \bZ_p)$. So it reduces to proving a natural equivalence between the categroy ${\rm Crys}(\calO_{\Prism}[\frac{1}{I}]^{\wedge}_p, \varphi)$ and the category $\Et\Phi(W(\widehat{\calO}_{X^{\flat}}))$. This follows from the next theorem.

	\end{proof}
	
	Note that the above equivalence can be checked locally on small affine opens and then glued up globally.	In fact, the association $\frakX\mapsto {\rm Crys}((\frakX)_{\Prism},\calO_{\Prism}[\frac{1}{I}]^{\wedge}_p, \varphi))$ is a stack for the Zariski topology (for example, see \cite[Theorem 5.17]{MT}). So we assume $\frakX={\Spf(R)}$  such that there exists a framing $\square:\calO_K\langle\underline{T}^{\pm1}\rangle\to R$.
	
	\begin{thm}
		There exists a commutative diagram of categories
		
		\begin{equation}
			\xymatrix@C=0.5cm{
				& \ar[ld]_{\alpha}	{\rm Crys}(\calO_{\Prism}[\frac{1}{I}]^{\wedge}_p, \varphi)\ar[rd]^{\rm ev_{\infty}}&\\
				\Et\Phi(W(\widehat{\calO}_{X^{\flat}}))\ar[rr]_{\Gamma(U_{\infty}, -)}&&\Et\Phi G(\widehat{\Ainf(\widehat{\overline R})[\frac{1}{\xi}]})
				,}
		\end{equation}
		where all the functors are equivalences. Here $U_{\infty}$ is the affinoid perfectoid object corresponding to $(\widehat{\overline R}[\frac{1}{p}], \widehat{\overline R})$ as in Remark \ref{Rmk-General groups}. The functor $\rm ev_{\infty}$ is the evaulation on the object $(\widehat{\Ainf(\widehat{\overline R})[\frac{1}{\xi}]}, (\xi))$.
	\end{thm}
	
	\begin{proof}
		
		We first construct the functor $\alpha$ in a natural way and prove it is an equivalence by studying this commutative diagram.
		
		Let $\calM$ be a prismatic $F$-crystal in $	{\rm Crys}(\calO_{\Prism}[\frac{1}{I}]^{\wedge}_p, \varphi)$.  Note that the site $X_{\proet}$ admits a basis consisting of affinoid perfectoid objects. To define a sheaf on $X_{\proet}$ then is equivalent to defining a sheaf on the site consisting of affinoid perfectoid objects with the induced topology. We can define a pro-\'etale presheaf $\alpha(\bM)$ on $X_{\proet}$ corresponding to $\calM$ as follows: let $V=``\varprojlim"_{i\in I}V_i$ be affinoid perfectoid with $V_i=\Spa(S_i, S_i^+)$ and $S^+=(\varinjlim_iS_i^+)^{\wedge}$  being a perfectoid ring and $\hat V$ be the associated perfectoid space. We can define
		\[
		\alpha(\calM)(V):=\calM((W((S^+)^{\flat}), (\Ker(\theta))).
		\]
		
		By the fact that there is an equivalence between the two sites $\hat V_{\proet}$ and $X_{\proet}/V$ and \cite[Corollary 9.3.8]{KL}, we see that for any \'etale $W(\hat \calO_{X^{\flat}})$-module $\calE$ on $X_{\proet}$, its restriction on $X_{\proet}/V$ is isomorphic to $\calE(V)\otimes_{W(\hat \calO_{X^{\flat}})(V)}W(\hat \calO_{X^{\flat}})|_V$.  This means $\alpha(\calM)|_V$ is a already a sheaf on $X_{\proet}/V$ for any affinoid perfectoid object $V$ by the definition of prismatic crystals.
		
		Then it is easy to see that $\alpha(\calM)$ is in $	\Et\Phi(W(\widehat{\calO}_{X^{\flat}}))$ and the diagram in the theorem is commutative by the construction.
		Now by Corollay \ref{descend equivalence}, the functor $\rm ev_{\infty}$ is an equivalence. Then ${\rm ev_{\infty}}^{-1}\circ \Gamma(U_{\infty},-)$ is a left inverse to the functor $\alpha$.
		
		To show $\alpha$ is an equivalence, it remains to show that $\alpha( {\rm ev_{\infty}^{-1}}( \Gamma(U_{\infty},\calE)))\cong \calE$ for any $\calE\in \Et\Phi(W(\widehat{\calO}_{X^{\flat}}))$. Let $\calM={\rm ev_{\infty}^{-1}}( \Gamma(U_{\infty},\calE))$. For any affinoid perfectoid objects $V\in X_{\rm pro\acute et}$, write $V_{\infty}:=U_{\infty}\times_XV$. Then $\calE(V)=\ker(\calE(V_{\infty})\to \calE(V_{\infty}\times_VV_{\infty}))$. By abuse of notation, let $\calM(V)$ denote the value of $\calM$ on the prism corresponding to the affinoid perfectoid space $\hat V$. Then we also have $\calM(V)=\ker(\calM(V_{\infty})\to \calM(V_{\infty}\times_VV_{\infty}))$. Since both $V_{\infty}$ and $V_{\infty}\times_VV_{\infty}$ live over $U_{\infty}$, we get $\calE(V)=\alpha(\calM)(V)=\calM(V)$. We are done.
	\end{proof}
	
	\begin{rmk}
	    By similar arguments, one can also prove that there is an equivalence between the category $\Crys(\overline{\calO}_{\Prism}[\frac{1}{p}])$ of rational Hodge--Tate crystals over $(\frakX)_{\Prism}$ and the category ${\rm Vect}(X_{\proet},\widehat{\calO}_X)$ of generalized representations over $X_{\proet}$.
	\end{rmk}

	\section{\' Etale comparison}
	
	In this section, we fix a perfect prism $(A,I=(d))$ such that $A/I$ contains all $p$-power roots of unity, which ensures that the arguments in Section \ref{local} can be safely applied. Note that for any $p$-adic formal scheme $\frakX$ over $A/I$ we have $(\frakX/A)_{\Prism}\simeq \frakX_{\Prism}$. Under the equivalence of categories in Section \ref{correspondence}, we want to study the relationship between the prismatic cohomology of crystals and the pro-\' etale cohomology of the corresponding $\Zp$-local systems.  The main theorem of this section is the following.

	\begin{thm}\label{global etale}
		Let $\frakX$ be a separated $p$-adic smooth formal scheme over $A/I$ and $\calM$ be a prismatic $F$-crystal over $\calO_{\Prism}[\frac{1}{I}]^{\wedge}_p$ with corresponding $\Zp$-local system $\calL$ on $X_{\rm \acute et}$. Then there is a quasi-isomorphism
		\[
		R\Gamma((\frakX/A)_{\Prism}, \calM)^{\varphi=1}\simeq R\Gamma(X_{\rm \acute et}, \calL).
		\]
	\end{thm}

	We check this theorem locally in a good functorial way. Assume $\frakX=\Spf(R)$ where $R$ admits an \'etale map from $A/I\langle T_1^{\pm 1},\cdots, T_n^{\pm 1}\rangle$ for some $n$. Our strategy is to relate prismatic cohomology on $(R/A)_{\Prism}$ to that on $(R/A)_{\Prism}^{\rm perf}$. The latter is more closely related to the \'etale cohomology of the generic fiber $\Spa(R[\frac{1}{p}], R)$.
	
	We first prove some lemmas showing that the  prismatic cohomology of crystals can also be calculated by some $\rm \check{C}$ech-Alexander complexes. The following lemma is an analogue of \cite[Corollary 3.12]{BS}.



	\begin{lem}\label{vanish}
		Let $(B,(d))$ be a transversal prism in $(R/A)_{\Prism}$, i.e. $(p, d)$ is a regular sequence in $B$. Then the higher cohomology of $\calO_{\Prism}[\frac{1}{I}]^{\wedge}_p$ on $(B, (d))$ vanishes.
	\end{lem}
	
	\begin{proof}
		Let $(B,(d))\to (C,(d))$ be a flat cover. Then $C\otimes_{B}^{\bL}B/{(p,d)}$ is concentrated in degree 0 by definition. Since $(p,d)$ is a regular sequence in $B$, $(p,d)$ is also a regular sequence in $C$. In particular, $C$ is $p$-torsion free. Let $(C^i,(d))$ be $(i+1)$-fold fiber product $(C,(d))\times_{(B,(d))}\cdots\times_{(B,(d))}(C,(d))$ of $(C,(d))$ over $(B,(d))$. By the same argument,  $C^i$ is also $p$-torsion free for any $i\geq 0$.
		
		On the other hand, we have $B \simeq R\varprojlim C^i$. Inverting $I$, we have $B[\frac{1}{I}]\simeq R\varprojlim C^i[\frac{1}{I}]$. By taking derived $p$-adic completion, we have
		\[
		B[\frac{1}{I}]^{\wedge}_{p,\rm derived}\simeq R\varprojlim C^i[\frac{1}{I}]^{\wedge}_{p,\rm derived}.
		\]
		Since $B[\frac{1}{I}]$ and the $C^i[\frac{1}{I}]$'s are all $p$-torsion free, their derived $p$-adic completion is the same as their classical $p$-adic completion.  So we have
		\[
		B[\frac{1}{I}]^{\wedge}_{p}\simeq R\varprojlim C^i[\frac{1}{I}]^{\wedge}_{p}.
		\]
		This means all the higher $\rm \check C$ech cohomology groups vanish, which implies
		\[
		R\Gamma((B,(d)), \calO_{\Prism}[\frac{1}{I}]^{\wedge}_p)=\Gamma((B, (d)),\calO_{\Prism}[\frac{1}{I}]^{\wedge}_p).
		\]
	\end{proof}
	
	\begin{cor}\label{vanishing for crystals}
		Let $(B,(d))$ be a transversal prism in $(R/A)_{\Prism}$. Then for any prismatic $F$-crystal $\calM$ over $\calO_{\Prism}[\frac{1}{I}]^{\wedge}_p$, we have
		\[
		R\Gamma((B,(d)), \calM|_{(B,(d))})=\Gamma((B, (d)),\calM).
		\]
	\end{cor}
	
	\begin{proof}
		By definition, we see that $\calM(B,(d))$ is a finite projective module over $B[\frac{1}{I}]^{\wedge}_p$ and $\calM|_{(B,(d))}=\calM(B,(d))\otimes_{B[\frac{1}{I}]^{\wedge}_p} \calO_{\Prism}[\frac{1}{I}]^{\wedge}_p$. By Lemma \ref{vanish} and the projection formula
		\[
		R\Gamma((B,(d)), \calM(B,(d))\otimes^{\bL}_{B[\frac{1}{I}]^{\wedge}_p} \calO_{\Prism}[\frac{1}{I}]^{\wedge}_p)\simeq R\Gamma((B,(d)),\calO_{\Prism}[\frac{1}{I}]^{\wedge}_p)\otimes^{\bL}_{B[\frac{1}{I}]^{\wedge}_p}\calM(B,(d)),
		\]
		we are done.
	\end{proof}
	
	Let's choose a framing, i.e. an \'etale map $\square: A/I\langle T_1^{\pm 1},\cdots, T_n^{\pm 1}\rangle\to R$. There is a unique lift $A^{\square}(R)$ of $R$ to $A$. Lemma \ref{Lem-Special cover} tells us that $(A^{\square}(R),(d))$ is a cover of the final object of the topos ${\rm Shv}((R/A)_{\Prism})$.
	
	\begin{prop}\label{Cech}
		
		For any prismatic $F$-crystal $\calM$ over $\calO_{\Prism}[\frac{1}{I}]^{\wedge}_p$, there is a quasi-isomorphism
		\[
		R\Gamma((R/A)_{\Prism}, \calM)\simeq R\varprojlim_i\Gamma(((A^{\square}(R)^i,(d)),\calM),
		\]
		where $(A^{\square}(R)^i,(d))$ is the $(i+1)$-fold self-product $(A^{\square}(R),(d))\times\cdots\times(A^{\square}(R),(d))$ of $(A^{\square}(R),(d))$.
	\end{prop}
	
	\begin{proof}
		Note that all $A^{\square}(R)^i$'s are $p$-torsion free. Since $(A^{\square}(R),(d))$ is a cover of the final object in the topos ${\rm Shv}((R/A)_{\Prism})$, we have \[
		R\Gamma((R/A)_{\Prism},\calM)\simeq R\varprojlim_iR\Gamma((A^{\square}(R)^i,(d)), \calM).
		\]
		
		Then this proposition follows from Corollary \ref{vanishing for crystals}.
	\end{proof}
	
	We have similar results for the perfect prismatic site. Recall that the prism $(\Ainf(R_{\infty}),(d))$, which is a cover of the final object of the topos ${\rm Shv}((R/A)_{\Prism}^{\rm perf})$, is the perfection of $(A^{\square}(R), (d))$.
	
	\begin{prop}\label{Infinity cech nerve}
		For any prismatic $F$-crystal $\calM$ over $\calO_{\Prism}[\frac{1}{I}]^{\wedge}_p$, there is a quasi-isomorphism
		\[
		R\Gamma((R/A)_{\Prism}^{\rm perf}, \calM)\simeq R\varprojlim_i\Gamma((\Ainf(R_{\infty})^i,(d)),\calM),
		\]
		where $(\Ainf(R_{\infty})^i,(d))$ is the $(i+1)$-fold self-product of $(\Ainf(R_{\infty}),(d)))$ in $(R/A)^{\rm perf}_{\Prism}$.
	\end{prop}
	
	\begin{proof}
		This follows from the same argument as in Proposition \ref{Cech}.
	\end{proof}
	
	Next we compare $R\Gamma((R/A)_{\Prism}, \calM)$ to $R\Gamma((A/R)_{\Prism}^{\rm perf}, \calM)$.
	
	\begin{thm}
		There is a quasi-isomorphsim
		\[
		R\Gamma((R/A)_{\Prism}, \calM)^{\varphi=1}\simeq R\Gamma((R/A)_{\Prism}^{\rm perf},\calM)^{\varphi=1}.
		\]
	\end{thm}
	
	\begin{proof}
		Note that $R\Gamma((R/A)_{\Prism},\calM)^{\varphi=1}$ is calculated by the total complex of the following bicomplex 
		
		\begin{equation*}
			\xymatrix@C=0.45cm{
				\Gamma((A^{\square}(R), (d)),\calM)       \ar[r]                     &  \Gamma((A^{\square}(R)^1, (d)),\calM)      \ar[r]                               &  \cdots  & \\
				\Gamma((A^{\square}(R), (d)),\calM)        \ar[u]^{\varphi-1}\ar[r]                 &      \Gamma((A^{\square}(R)^1, (d)),\calM)       \ar[u]^{\varphi-1}\      \ar[r]                             & \cdots   \ar[u]     &
			}
		\end{equation*}
		in the first quadrant. We denote this bicomplex by $C^{\bullet,\bullet}$.
		There ia a spectral sequences $E_1^{i,j}$ associated with this bicomplex, i.e. $E_0^{i,j}=C^{i,j}$. In particular, we have
		\[
		E_1^{i,1}=\Gamma((A^{\square}(R)^i,(d)),\calM)/{\Ima(\varphi-1)},
		\]
		\[
		E_1^{i,0}=\Gamma((A^{\square}(R)^i,(d)),\calM)^{\varphi=1}.
		\]
		
		For the perfect prismatic site, we have a similar bicomplex $C_{\rm perf}^{\bullet,\bullet}$, whose totalization calculates $R\Gamma((R/A)_{\Prism}^{\rm perf},\calM)^{\varphi=1}$, and a spectral sequence $E^{i,j}_{0,\rm perf}=C^{i,j}_{\rm perf}$. Clearly, there is a natural map from $C^{\bullet,\bullet}$ to $C^{\bullet,\bullet}_{\rm perf}$. Since in the both cases, the two spectral sequences converge to the cohomologies of the totalizations of corresponding bicomplexes. This means that we just need to compare the $E_1$-page $E_1^{i,j}$ and $E_{1,\rm perf}^{i,j}$. Then this theorem follows from the next lemma.
		
	\end{proof}

	\begin{lem}
		For any $i$, we have
		\[
		\Gamma((A^{\square}(R)^i,(d)),\calM)/{\Ima(\varphi-1)}\cong \Gamma((\Ainf(R_{\infty})^i,(d)),\calM)/{\Ima(\varphi-1)}
		\]
		and
		\[
		\Gamma((A^{\square}(R)^i,(d)),\calM)^{\varphi=1}\cong \Gamma((\Ainf(R_{\infty})^i,(d)),\calM)^{\varphi=1}.
		\]
	\end{lem}
	
	\begin{proof}
		Since $(p,(d))$ is a regular sequence in $A^{\square}(R)$, by \cite[Proposition 3.13]{BS} it is also regular in $A^{\square}(R)^i$ for any $i$ ( in fact, $A^{\square}(R)^i/(d)$ is the $p$-adic completion of a free PD-polynomial ring by \cite{Tian}). As mentioned in Remark \ref{Rmk-Preserve extension}, the equivalence \[\Et\Phi\rM(A^{\square}(R)^i[\frac{1}{I}]^{\wedge}_p)\rightarrow \Et\Phi\rM(\Ainf(R_{\infty})^i[\frac{1}{I}]^{\wedge}_p)\]
		preserves extensions.
		
		Let $E_1$ and $E_2$ be extensions of $N$ by $M $ for $\varphi$-modules $N$ and $M$ over some ring $A$. In other words, for $i = 1,2$, we have
		\[\xymatrix@C=0.45cm{
			0\ar[r]& M \ar[r]^{\alpha_i}& E_i\ar[r]^{\pi_i}& N\ar[r]& 0.
		}\]
		Then a morphism $f:E_1\rightarrow E_2$ is an isomorphism of extensions if and only if
		\[f\circ\alpha_1 = \alpha_2: M\rightarrow E_2\] and
		\[\pi_1 = \pi_2\circ f: E_1\rightarrow N.\]
		
		As a consequence, let $N,M\in \Et\Phi\rM(A^{\square}(R)^i[\frac{1}{I}]^{\wedge}_p)$ with associated $N_{\infty}, M_{\infty}\in \Et\Phi\rM(\Ainf(R_{\infty})^i[\frac{1}{I}]^{\wedge}_p)$. Then the above arguments imply that the set of isomorphism classes of extensions of $N$ by $M$ coincides with the one of $N_{\infty}$ by $M_{\infty}$. For the same reason, fix an extension $E$ of $N$ by $M$ and denote by $E_{\infty}$ the corresponding extension of $N_{\infty}$ by $M_{\infty}$, then the group of automorphisms of $E$ (as an extension) coincides with the one of $E_{\infty}$.
		
		Now the result follows from Lemma \ref{Lem-Cohomology of phi modules}.
	\end{proof}
	
	The following lemma on $\varphi$-modules was used above.
	\begin{lem}\label{Lem-Cohomology of phi modules}
		Let $A$ be a $\delta$-ring and $M$ be a $\varphi$-module over $A$. Then the set of isomorphic classes of extensions of $A$ by $M$ forms an $\rH^1_{\varphi}(M)$-torsor and the group of automorphisms of such a fixed extension is $\rH^0_{\varphi}(M)$.
	\end{lem}
	\begin{proof}
		Let $E$ be an extension of $A$ by $M$ with the underlying $A$-module $E = M\oplus A$. Then $E$ is uniquely determined by an element $m\in M$ satisfying $\varphi(0,1)=(m,1)$. For this reason, we denote $E$ by $E_m$. Let $f:E_m\rightarrow E_{m'}$ be an isomorphism of extensions. Then $f$ is uniquely determined by some $n\in M$ such that $f(0,1)=(n,1)$. Since $\varphi\circ f = f\circ\varphi$, we deduce that $m = m'+(\varphi-1)(n)$. This implies the first assertion. If moreover $m = m'$, that is, $f$ is an automorphism of $E_m$ (as an extension), then $n\in \rH^0_{\varphi}(M)$. This implies the second assertion.
	\end{proof}
	\begin{rmk}
		The description of $\varphi$-coinvariants in terms of extensions also appears in \cite[Chapter 11]{FF}.
	\end{rmk}
	
	Next we relate perfect prismatic cohomology to pro-\'etale cohomology. Let $X={\rm Spa}(R[\frac{1}{p}],R)$ and $U=``\varprojlim_i"{\rm Spa}(R_i[\frac{1}{p}],R_i)\to X$ be the usual pro-\'etale covering of $X$, where $R_i=R\otimes_{\square, A/I\langle \underline T^{\pm 1}\rangle}A/I\langle \underline T^{\pm \frac{1}{p^i}}\rangle$.
	
	\begin{thm}
		For any prismatic $F$-crystal $\calM$ over $\calO_{\Prism}[\frac{1}{I}]^{\wedge}_p$ with $\calE$ the corresponding $W({\hat \calO}_{X^{\flat}})$-module, there is a quasi-isomorphism
		\[
		R\Gamma((R/A)_{\Prism}^{\rm perf},\calM)\simeq R\Gamma(X_{\rm pro\acute et},\calE).
		\]
	\end{thm}

	\begin{proof}
		Let $(\Ainf(R_{\infty}),(d)))$ be the usual perfect prism which is a cover of the final object of the topos ${\rm Shv}((R/A)_{\Prism}^{\rm perf})$ by Lemma \ref{Lem-Special cover}. Then by Proposition \ref{Infinity cech nerve}, we have
		\[
		R\Gamma((R/A)_{\Prism}^{\rm perf},\calM)\simeq R\varprojlim_i\Gamma((\Ainf(R_{\infty})^i,(d))),\calM).
		\]
		
		On the other hand, we also have
		\[
		R\Gamma((X_{\rm pro\acute et},\calE)\simeq R\varprojlim_iR\Gamma(U^i,\calE).
		\]
		where $U^i$ is the $(i+1)$-fold self-product of $U$ in $X_{\rm pro\acute et}$. By \cite[Lemma 9.3.4]{KL} and the fact that $\calE|_{U^i}=\calE(U^i)\otimes_{W({\hat \calO}_{X^{\flat}})(U^i)}W({\hat \calO}_{X^{\flat}})|_{U^i}$, we have $R\Gamma(U^i,\calE)\simeq \Gamma (U^i,\calE)$. Then this theorem follows from $\Gamma((\Ainf(R_{\infty})^i,(d)),\calM)\cong \Gamma (U^i,\calE)$, which is proved in the next paragraph.
		
		On one hand, let $\hat U^i$ be the affinoid perfectoid space associated with $U^i$. Then $\hat U^i=\hat U\times \underline{\bZ_p}^i$. On the other hand, as the subcategory of perfect prisms in $(R)_{\Prism}$ is equivalent to the category ${\rm Perfd}_R$ of perfectoid rings over $R$, we see that $R_{\infty}^i:=A_{\inf}(R_{\infty})^i/(d)$ is isomorphic to the $(i+1)$-fold self-product of $R_{\infty}$ in the category ${\rm Perfd}_R$. By checking the universal property, we see that ${\rm Spa}(R_{\infty}^i[\frac{1}{p}],R_{\infty}^i[\frac{1}{p}]^{+})$ is  the $(i+1)$-fold self-product of ${\rm Spa}(R_{\infty}[\frac{1}{p}],R_{\infty})$ in the category ${\rm Perfd}_X$ of affinoid perfectoid spaces over $X$, which is equivalent to the category ${\rm Perfd}_{X^{\diamond}}$. Here $R_{\infty}^i[\frac{1}{p}]^{+}$ is the $p$-adic completion of the integral closure of the image of $S_{\infty}^i$ in $R_{\infty}^i[\frac{1}{p}]$.
		Since ${\rm Spa}(R_{\infty}[\frac{1}{p}],R_{\infty})^{\diamond}$ is a $\underline {\bZ_p}$-torsor over ${\rm Spa}(R[\frac{1}{p}],R)^{\diamond}$, we see that the $(i+1)$-fold fiber product ${\rm Spa}(R_{\infty}[\frac{1}{p}],R_{\infty})^{\diamond}\times_{{\rm Spa}(R[\frac{1}{p}],R)^{\diamond}}\times\cdots\times_{{\rm Spa}(R[\frac{1}{p}],R)^{\diamond}}{\rm Spa}(R_{\infty}[\frac{1}{p}],R_{\infty})^{\diamond}$ is representable by an affinoid perfectoid space ${\rm Spa}(R_{\infty}[\frac{1}{p}],R_{\infty})^{\flat}\times \underline{\bZ_p}^i$. This implies $\hat U^i\simeq {\rm Spa}(R_{\infty}^i[\frac{1}{p}],R_{\infty}^i[\frac{1}{p}]^{+})$.  Since $W((R^i_{\infty})^{\flat})[\frac{1}{d}]^{\wedge}_{p}\cong W((R_{\infty}^i[\frac{1}{p}]^{\flat})\cong W((R_{\infty}^i[\frac{1}{p}]^{+}))^{\flat})[\frac{1}{d}]^{\wedge}_{p}$, we have $\Gamma((\Ainf(R_{\infty})^i,(d)),\calM)\cong \Gamma (U^i,\calE)$.
	\end{proof}

	\begin{thm}
		There is a natural quasi-isomorphism
		\[
		R\Gamma((R/A)_{\Prism},\calM)^{\varphi=1}\simeq R\Gamma({\rm Spa}(R[\frac{1}{p}],R),\calL)
		\]
		where $\calL$ is the \'etale $\bZ_p$-local system associated with the prismatic $F$-crystal $\calM$.
	\end{thm}
	\begin{proof}
		For each  framing $\square: A/I\langle \underline T^{\pm 1}\rangle\to R$, we have quasi-isomorphisms
		\[
		R \Gamma((R/A)_{\Prism}^{\rm perf},\calM)\simeq  \Gamma((\Ainf(R^{\square}_{\infty})^{\bullet },(d)),\calM)\simeq R\Gamma(X_{\rm pro\acute et},\calE)
		\]
		
		Consider the following index set
		\[
		J=\{S\mid S\ {\rm is\ a\ finite\ set \ of \ framings \ of }\ R\}.
		\]
		
		For each $S=\{ \square_1,\cdots, \square_n\}$, we set
		\[\
		(\Ainf(R_{\infty}^{\square_S}),(d))=(\Ainf(R_{\infty}^{\square_1}),(d))\times \cdots\times (\Ainf(R_{\infty}^{\square_n}),(d))
		\]
		in $(R/A)_{\Prism}^{\rm perf}$.
		
		Since $(\Ainf(R_{\infty}^{\square_S}),(d))$ also covers the final object of ${\rm Shv}((R/A)_{\Prism})$, we have
		\[
		R \Gamma((R/A)_{\Prism}^{\rm perf},\calM)\simeq  \Gamma((\Ainf(R^{\square_S}_{\infty})^{\bullet},(d)),\calM)\simeq R\Gamma(X_{\rm pro\acute et},\calE).
		\]
		If $S_1\subset S_2$, then $\Gamma((\Ainf(R^{\square_{S_1}}_{\infty})^{\bullet},(d)),\calM)\simeq \Gamma((\Ainf(R^{\square_{S_2}}_{\infty})^{ \bullet},(d)),\calM)$. Now we consider the filtered colimit
		\[
		M:=\varinjlim_{S\in J}\Gamma((\Ainf(R^{\square_S}_{\infty})^{\bullet},(d)),\calM)
		\]
		which then does not depend on any framing. Then we get natural quasi-isomorphisms
		\[
		R \Gamma((R/A)_{\Prism}^{\rm perf},\calM)\simeq  M\simeq R\Gamma(X_{\rm pro\acute et},\calE).
		\]
		
		Since the quasi-isomorphsims $R\Gamma((R/A)_{\Prism},\calM)^{\varphi=1}\simeq R\Gamma((R/A)_{\Prism}^{\rm perf},\calM)^{\varphi=1}$ and $R\Gamma(X_{\rm pro\acute et},\calE)^{\varphi=1}\simeq R\Gamma({\rm Spa}(R[\frac{1}{p}],R),\calL)$ are already natural, we get the desired natural quasi-isomorphism
		\[
		R\Gamma((R/A)_{\Prism},\calM)^{\varphi=1}\simeq R\Gamma({\rm Spa}(R[\frac{1}{p}],R),\calL).
		\]
	\end{proof}
	
	Now the Theorem \ref{global etale} is just the globalization of the affine case via these natural quasi-isomorphisms.
	
	\begin{rmk}
		Our proof can not apply to the case of arbitrary $p$-adic formal scheme as in \cite{BS}. In fact, our local arguments highly depend on a good cover of the final object of the topos ${\rm Shv((\frakX/A)_{\Prism})}$, which enables us to use $\rm\check{C}$ech-Alexander complex to calculate the prismatic cohomology of crystals. One might also try to study arc-descent results in the non-constant coefficient  case.
	\end{rmk}

	By the same arguments, one can also deduce the following theorem.
	\begin{thm}\label{O_K}
	    	Let $\frakX$ be a separated $p$-adic smooth formal scheme over $\calO_K$ and $\calM$ be a prismatic $F$-crystal over $\calO_{\Prism}[\frac{1}{I}]^{\wedge}_p$ on $\frakX_{\Prism}$ with corresponding $\Zp$-local system $\calL$ on $X_{\rm \acute et}$. Then there is a quasi-isomorphism
		\[
		R\Gamma((\frakX)_{\Prism}, \calM)^{\varphi=1}\simeq R\Gamma(X_{\rm \acute et}, \calL).
		\]
	\end{thm}
	
	\begin{proof}
	The proof is the same as that of Theorem \ref{global etale}.
	\end{proof}
	
	


\begin{thebibliography}{99}
		\bibitem[BMS18]{BMS-a} B. Bhatt, M. Morrow and P. Scholze: {\it Integral $p$-adic Hodge theory}, Publ. math. de l'IH\'ES 128, Issue 1, pp. 219-395, (2018).https://doi.org/10.1007/s10240-019-00102-z.
		
		\bibitem[BMS19]{BMS-b} B. Bhatt, M. Morrow and P. Scholze: {\it Topological Hochschild homology and integral $p$-adic Hodge theory}, Publ. math. de l'IH\'ES 129, Issue 1, pp. 199-310, (2019). https://doi.org/10.1007/s10240-019-00106-9.
		
		\bibitem[BS19]{BS} B. Bhatt and P. Scholze: {\it Prisms and prismatic cohomology}, arXiv:1905.08229v3, (2019).
		
		
		\bibitem[BS21]{BS1} B. Bhatt and P. Scholze: {\it Prismatic $ F $-crystals and crystalline Galois representations}, arXiv:2106.14735, (2021).
		
			\bibitem[FF18]{FF}L. Fargues and J.M. Fontaine:{\it Courbes et fibrés vectoriels en théorie de Hodge $p$-adique}		, Ast\'erisque 406, (2018).
		
		\bibitem[KL15]{KL} K.S. Kedlaya and R. Liu: {\it Relative $p$-adic Hodge theory: Foundations}, Ast\'erisque 371, (2015).
			\bibitem[Moon18]{Moon}Y. S, Moon: {\it Relative crystalline representations and weakly admissible modules},  arXiv:1806.00867,(2018).
		

		
		\bibitem[MT20]{MT} M. Morrow and T. Tsuji: {\it Generalised representations as $q$-connections in integral $p$-adic Hodge theory}, arXiv:2010.04059v1, (2020).
		
		\bibitem[Sch17]{Sch-a} P. Scholze: {\it \'Etale cohomology of diamonds}, arXiv:1709.07343v2, (2017).
		
		\bibitem[SW]{SW} P. Scholze and J. Weinstein: {\it Berkeley lectures on $p$-adic geometry}, Princeton University Press, Princeton, NJ, (2020).
		
		\bibitem[Tian21]{Tian} Y. Tian: {\it Finiteness and duality for the cohomology of prismatic crystals}, arXiv:2109.00801v1, (2021).
		
		\bibitem[Wu21]{Wu} Z. Wu: {\it Galois representations, $(\varphi,\Gamma)$-modules and prismatic $F$-crystals}, arXiv:2104.12105v3, (2021).
		
	
		
	
	\end{thebibliography}
\end{document}